\documentclass[12pt]{amsart}
\pdfoutput=1 

\usepackage{amssymb,amsmath,amsfonts,amsthm}
\usepackage{graphics,mathrsfs,mathtools,tikz-cd,soul,csquotes}
\usetikzlibrary{positioning,arrows,scopes}
\usepackage{bm,dsfont}
\usepackage[color=red!40]{todonotes}

\usepackage[normalem]{ulem}
\usepackage[left=3.5cm,top=3.5cm,right=3.5cm]{geometry}
\usepackage[shortlabels]{enumitem}
\usepackage[T1]{fontenc}
\usepackage{lmodern}
\usepackage{microtype}

\usepackage{libertinus}
\usepackage[T1]{fontenc}

\definecolor{cadmiumgreen}{rgb}{0.0, 0.42, 0.24}
\usepackage[
	colorlinks, citecolor=cadmiumgreen,
	pdfauthor={Harry Richman, Farbod Shokrieh, Chenxi Wu}, 
	pdfstartview ={FitV},
]{hyperref}
\hypersetup{
pdftitle={Principal minors of tree distance matrices},
}

\usepackage[
	alphabetic,
	msc-links,
	nobysame,
	lite,
]{amsrefs} 

\newtheorem{thm}{Theorem}[section]
\newtheorem*{thm*}{Theorem}
\newtheorem{mainthm}{Theorem}

\newtheorem{prop}[thm]{Proposition}
\newtheorem{lem}[thm]{Lemma}
\newtheorem{cor}[thm]{Corollary}

\theoremstyle{definition}

\newtheorem{dfn}[thm]{Definition}
\newtheorem*{dfn*}{Definition}
\newtheorem{eg}[thm]{Example}
\newtheorem*{eg*}{Example}
\newtheorem{rmk}[thm]{Remark}
\newtheorem*{rmk*}{Remark}

\newcommand{\RR}{\mathbb{R}}

\newcommand{\one}{\mathds{1}}
\newcommand{\bone}{\mathbf{1}}

\newcommand{\boldm}{\mathbf{m}}

\newcommand{\boldu}{\mathbf{u}}
\newcommand{\boldw}{\mathbf{w}}
\newcommand{\Da}{{D}}
\newcommand{\La}{L}
\newcommand{\tr}{\intercal}
\DeclareMathOperator{\cof}{cof}
\DeclareMathOperator{\conv}{conv}

\newcommand{\trees}{\mathcal{F}_1}
\newcommand{\forests}{\mathcal{F}}

\DeclareMathOperator{\outdeg}{outdeg}

\begin{document}
\title{Principal minors of tree distance matrices}
\author{Harry Richman}
\email{\href{mailto:hrichman@ncts.ntu.edu.tw}{hrichman@ncts.ntu.edu.tw}}
\author{Farbod Shokrieh}
\email{\href{mailto:farbod@uw.edu}{farbod@uw.edu}}
\author{Chenxi Wu}
\email{\href{mailto:cwu367@math.wisc.edu}{cwu367@math.wisc.edu}}
\date{\today}

\subjclass[2020]
{
\href{https://mathscinet.ams.org/msc/msc2020.html?t=05C50}{05C50},
\href{https://mathscinet.ams.org/msc/msc2020.html?t=05C05}{05C05},
\href{https://mathscinet.ams.org/msc/msc2020.html?t=05C12}{05C12},
\href{https://mathscinet.ams.org/msc/msc2020.html?t=05C30}{05C30},
\href{https://mathscinet.ams.org/msc/msc2020.html?t=31C15}{31C15}
}


\begin{abstract}
We prove that the principal minors of the distance matrix of a tree satisfy a combinatorial expression involving counts of rooted spanning forests of the underlying tree.
This generalizes a result of Graham and Pollak,
and refines a result of Graham and Lov\'asz on the coefficients of the characteristic polynomial of the distance matrix.
We also give such an expression for the case of trees with edge lengths.
We use arguments motivated by potential theory on graphs.
Our formulas can be expressed in terms of evaluations of Symanzik polynomials.
\end{abstract}
\maketitle

\setcounter{tocdepth}{1}
\tableofcontents

\section{Introduction}
\renewcommand*{\themainthm}{\Alph{mainthm}}

Suppose $G = (V,E)$ is a tree with $n$ vertices.
Let $D$ denote the {\em distance matrix} of $G$, defined by setting the $(u,v)$-entry to the length of the unique path from $u$ to $v$.
In~\cite{graham-pollak}, Graham and Pollak prove
\begin{equation}\label{eq:full-det}
	\det D = (-1)^{n-1} 2^{n-2} (n-1). 
\end{equation}
This identity is remarkable in that the result does not depend on the tree structure,
beyond the number of vertices.
The identity \eqref{eq:full-det} was motivated by a problem in data communication,
and inspired much further research on distance matrices.

The main result of this paper is to generalize \eqref{eq:full-det} by replacing $\det D$ with any of its principal minors.
For a subset $S \subseteq V$, let $D[S]$ denote the submatrix consisting of the $S$-indexed rows and columns of $D$.
\begin{mainthm}
\label{thm:main}
	Suppose $G$ is a tree with $n$ vertices whose distance matrix is $D$.
	Let $S \subseteq V$ be a nonempty subset of vertices.
	Then
	\begin{equation}\label{eq:main}
		\det D[S] = (-1)^{|S|-1} 2^{|S|-2} \Big( (n-1)\, \kappa(G;S)  - \!\!\! \sum_{F \in \forests_2(G; S)} \left(\outdeg(F,*) - 2\right)^2  \Big),
	\end{equation}
	where 
	$\kappa(G;S)$ is the number of $S$-rooted spanning forests of $G$,
	$\forests_2(G;S)$ is the set of $(S,*)$-rooted spanning forests of $G$,
	and
	$\outdeg(F,*)$ denotes the outdegree of the floating component of $F$.
\end{mainthm}
For definitions of rooted spanning forests as well as other terminology, see \S\ref{sec:graphs-matrices}.
When $S = V$ is the full vertex set, the set of $V$-rooted spanning forests is a singleton (consisting of the subgraph with no edges) so $\kappa(G; V) = 1$. Moreover, the set $\forests_2(G; V)$ of $(V, *)$-rooted spanning forests is empty. 
Thus Theorem~\ref{thm:main} recovers the Graham--Pollak identity \eqref{eq:full-det}.

\subsection{Trees with edge lengths}

Assume the tree $G$ is endowed with a collection of positive real edge lengths $\alpha = \{\alpha_e \colon e\in E\}$. 
The {\em distance matrix} (with respect to $\alpha$), which we will again denote by $\Da$, is defined by setting the $(u,v)$-entry to the sum of the edge lengths $\alpha_e$ along the unique path from $u$ to $v$.
The relation \eqref{eq:full-det} has a generalization for the distance matrix of a tree with edge lengths, proved by Bapat, Kirkland, and Neumann (\cite{bapat-kirkland-neumann}*{Corollary 2.5}):
\begin{equation}\label{eq:w-full-det}
	\det \Da = (-1)^{n-1} 2^{n-2} \Big(\sum_{e \in E} \alpha_e\Big) \Big( \prod_{e \in E} \alpha_e \Big).
\end{equation}
We, in fact, prove the following more general result. 
For an edge subset $A \subseteq E$ 
let $w(A) = \prod_{e \not\in A} \alpha_e$.
(The quantity $w(A)$ is sometimes known as the ``weight'' of $A$, while other sources call it the ``coweight.'')
\begin{mainthm}\label{thm:w-main}
	Suppose $G = (V,E)$ is a finite tree with edge lengths $\{\alpha_e \colon e \in E\}$, and corresponding distance matrix $\Da$. 
	For any nonempty subset $S \subseteq V$, we have
	\begin{equation}\label{eq:w-main}
		\det \Da[S] = (-1)^{|S|-1} 2^{|S|-2} \Bigg( \sum_{e \in E}\alpha_e \!\!\!\! \sum_{T \in \trees(G;S)} \!\! w(T) - \!\!\!\!\sum_{F \in \forests_2(G;S)} \!\!w(F)\, (\outdeg(F,*) - 2)^2 \!\Bigg),
	\end{equation}
	where 
	$\trees(G;S)$ is the set of $S$-rooted spanning forests of $G$,
	and other notation is as above in Theorem~\ref{thm:main}.
\end{mainthm}
See \S\ref{sec:graphs-matrices} for explanations of the other terminology used here.
For examples of Theorem~\ref{thm:main} and \ref{thm:w-main} applied to specific trees, see \S\ref{sec:examples}.

\subsection{Coefficients of the characteristic polynomial}

In~\cite{graham-lovasz}, Graham and Lov\'asz study
the characteristic polynomial 
$
    \det(D - \lambda I) = \sum_{i \geq 0} \delta_i\, \lambda^i
$
of the distance matrix of a tree.
Note that $\delta_0 = \det D$.
They generalize the determinant formula \eqref{eq:full-det} by finding a combinatorial formula for every coefficient $\delta_i = \delta_i(G)$ of the characteristic polynomial.
Their expression for $\delta_i$ involves summing over all spanning forests of $G$ with $(i - 1)$, $i$, or $(i + 1)$-many edges.

The simplest case of their theorem, beyond $\delta_0$, is~\cite[Equation (6), p. 63]{graham-lovasz}
\begin{equation}\label{eq:delta-1-graham-lovasz}
    \delta_1(G) = (-1)^{n - 1} 2^{n - 3} \left(4\cdot N_{2 P_1}(G) + 2\cdot N_{P_2}(G) + 4n - 8 \right) ,
\end{equation}
where $G$ is a tree on $n$ vertices, and $N_{2 P_1}(G)$ (resp. $N_{P_2}(G)$) denotes the number of subgraphs of $G$ isomorphic to two disjoint edges (resp. to a two-edge path).

There is a well-known relationship between coefficients of the characteristic polynomial and principal minors, namely
\begin{equation*}
    \delta_i = (-1)^i \sum_{\substack{|S| = n - i}} \det D[S].
\end{equation*}
By applying Theorem~\ref{thm:main} and summing over vertex subsets of fixed size,
we obtain a new combinatorial expression for the coefficient $\delta_i$.

\begin{mainthm}
\label{thm:dist-poly-coeff}
Let $G$ be a tree with unit edge lengths, 
and let $\det(D - \lambda I) = \sum_{i} \delta_i \lambda^i$
denote the characteristic polynomial of its distance matrix.
Then 
\[
	\delta_{n - k} = (-1)^{n - 1} 2^{k - 2} \left((n - 1) \sum_{\substack{F \in \mathcal F_{k}(G) }} \!\!\! \pi(F) - \!\!\!\!\!\! \sum_{\substack{F \in \mathcal F_{k + 1}(G) \\ F = C_1 \sqcup \cdots \sqcup C_{k + 1}}} \!\!\!\!\!\! \pi(F) \sum_{j = 1}^{k + 1}  \frac{(\outdeg(C_j) - 2)^2}{|V(C_j)|}\right),
\]
where 
$\forests_k(G)$ denotes the set of $k$-component spanning forests of $G$,
and $\pi(F) = \prod_{i = 1}^k |V(C_i)|$ 
for a forest $F$ with components $C_1 \sqcup \cdots \sqcup C_k$.
\end{mainthm}

The expression in Theorem~\ref{thm:dist-poly-coeff} should be compared with the main result of Graham--Lov\'asz~\cite[Equation (42), p. 81]{graham-lovasz}, where the expression occupies three lines and it involves summing over spanning forests in $G$ with $(k - 1)$ components, as well as those with $k$ or $(k + 1)$ components.

It is not at all obvious that the Graham--Lov\'asz identity is equivalent to Theorem~\ref{thm:dist-poly-coeff}.
Here, we elaborate on the case of $\delta_1$: 
each forest in $\forests_{n - 1}(G)$ 
contains exactly one edge of $G$, so $|\forests_{n - 1}(G)| = n - 1$ and each forest $F$ in $\forests_{n - 1}(G)$ has $\pi(F) = 2$.
In $\forests_n(G)$ there is exactly one forest, with $\pi(F) = 1$, whose components are exactly the vertices of $G$.
Thus
\begin{equation}
	\delta_1(G) = (-1)^{n - 1} 2^{n - 3} \left( 2(n - 1)^2 - \sum_{v \in V} (\deg(v) - 2)^2 \right).
\end{equation}
Compare to \eqref{eq:delta-1-graham-lovasz}, which has a different combinatorial flavor.

\subsection{Normalized principal minors}\label{sec:Normalized principal minors}
One may ask how the expressions $\det D[S]$ vary, as we fix a tree distance matrix and vary the vertex subset $S$.
To our knowledge there is no nice behavior among these determinants but, as $S$ varies, there is nice behavior of a ``normalized'' version which we describe here.

Given a matrix $A$, let $\cof A$ denote the {\em sum of cofactors} of $A$, i.e. 
\begin{equation}\label{eq:cof}
	\cof A = \sum_{i} \sum_{j} (-1)^{i + j} \det A_{i,j} 
\end{equation}
where $A_{i,j}$ is the submatrix of $A$ that removes the $i$-th row and the $j$-th column.
If $A$ is invertible, then $\cof A$ is the sum of entries of the matrix inverse $A^{-1}$ multiplied by a factor of $\det A$, i.e. $\cof A = (\det A) (\bone^\tr A^{-1} \bone)$.

In \cite{bapat-sivasubramanian}, Bapat and Sivasubramanian show the following identity for the sum of cofactors of a distance submatrix $D[S]$ of a tree with edge lengths $\{\alpha_e \colon e \in E\}$,
\begin{equation}\label{eq:cof-trees}
	\cof D[S] = (-2)^{|S| - 1} \!\!\!\sum_{T \in \trees(G;S)} w(T) , 
\end{equation}
where $\trees(G;S)$ denotes the set of $S$-rooted spanning forests of $G$ (see \S\ref{sec:graphs-matrices} and Theorem~\ref{thm:distance-sub-cof}).
An immediate consequence of Theorem~\ref{thm:w-main}, together with \eqref{eq:cof-trees}, is the identity
\begin{equation}\label{eq:det-cof}
	\frac{\det D[S]}{\cof D[S]} = \frac12 \left( \sum_{e \in E} \alpha_e - \frac{\sum_{F \in \forests_2(G; S)} w(F) \,(\outdeg(F,*) - 2)^2}{\sum_{T \in \trees(G; S)} w(T)} \right).
\end{equation}

We will refer to ${\det D[S]}/{\cof D[S]}$ as the {\em normalized principal minor} corresponding to the vertex subset $S$. 
It turns out that normalized principal minors satisfy a monotonicity condition.

\begin{mainthm}[Monotonicity of normalized principal minors]
\label{thm:monotonic}
	In the setting of Theorem~\ref{thm:w-main}, if $A, B \subseteq V$ are nonempty vertex subsets with $A \subseteq B$,
	then
	\begin{equation*}
		\frac{\det D[A]}{\cof D[A]}  \leq \frac{\det D[B]}{\cof D[B]}.
	\end{equation*}
\end{mainthm}

The essential observation behind this result is an intriguing connection to the theory of quadratic optimization; 
we will show that the normalized principal minor $\det D[S] / \cof D[S]$ is the solution of the following optimization problem: for all vectors $ \boldu \in \RR^S$,
\begin{align*}
	\text{maximize objective function:} &\quad \boldu^\tr D[S] \boldu \\
	\text{with constraint:} &\quad \bone^\tr \boldu = 1.
\end{align*}
This result is obtained using the method of Lagrange multipliers,
and relies on knowledge of the signature of $D[S]$.
The proof is given in \S\ref{sec:optimization}.

For a nonempty subset $S$, equation \eqref{eq:det-cof} implies the upper bound in
\begin{equation*}
	0 \;\leq\; \frac{\det D[S]}{\cof D[S]} \;\leq\; \frac12 \sum_{e \in E(G)} \alpha_e ,
\end{equation*}
while the lower bound is implied by the quadratic optimization formulation.
We can use the monotonicity property to obtain refined bounds on $\det D[S] / \cof D[S]$, see Theorem~\ref{thm:det-cof-bounds}.

The monotonicity result in Theorem~\ref{thm:monotonic} was found independently by Devriendt \cite{devriendt-thesis}*{Property 3.38},
in the more general context of effective resistance matrices of graphs.

\renewcommand*{\thethm}{\arabic{section}.\arabic{thm}}

\begin{rmk}
	The terms appearing in the submatrix $D[S]$ only depend on the lengths of edges which lie on some path between vertices in $S$.
	Depending on the chosen subset $S \subseteq V$, these distances in $D[S]$ may ignore a large part of the ambient tree $G$.
	However, the expression in Theorem~\ref{thm:w-main} involves all edge lengths.

	To calculate $\det D[S]$ in this case, using Theorem~\ref{thm:w-main} ``efficiently,'' we could instead replace $G$ by the subtree  $\conv(S,G)$ consisting of the union of all paths between vertices in $S$,
	which we call the {\em convex hull} of $S$.
	However, the formulas as stated are true even without this replacement due to cancellation of terms.
\end{rmk}

\subsection{Connection with Symanzik polynomials}\label{sec:symanzik}
Recall, given a graph $G = (V, E)$, the {\em first Symanzik polynomial} is the homogeneous polynomial in edge-indexed variables $\underline{x} = \{x_e \colon e \in E\}$ defined by
\[
	\psi_G(\underline{x}) = \sum_{T \in \trees(G)} \prod_{e \not \in T} x_e ,
\]
where $\trees(G)$ denotes the set of spanning trees of $G$.

A ``momentum function'' is any $p\colon V \to \RR$ satisfying the constraint $\sum_{v \in V} p(v) = 0$. 
The {\em second Symanzik polynomial} is
\[
	\varphi_G(p; \underline{x}) = \sum_{F \in \forests_2(G)} \Big(\sum_{v \in F_1} p(v) \Big)^2 \prod_{e \not\in F} x_e ,
\]
where $\forests_2(G)$ is the set of two-component spanning forests of $G$ and $F_1$ denotes one of the components of $F = F_1 \sqcup F_2$.
Note that it does not matter which component we label as $F_1$, since the momentum constraint implies that
$\sum_{v \in F_1} p(v)  = - \sum_{v \in F_2} p(v)$.

We note that the expressions in Theorem~\ref{thm:w-main} and \eqref{eq:det-cof} are closely related to Symanzik polynomials.
Let $\psi_{(G/S)}$ and $\varphi_{(G/S)}$ denote the first and second Symanzik polynomials of the quotient graph $G/S$.
Let $p_{\rm can}$ be the momentum function defined on $G/S$ by $p_{\rm can}(v) = \deg(v) - 2$ for $v \not \in S$ and $p_{\rm can}(v_S) = -\sum_{v \not\in S} p_{\rm can}(v)$, where $v_S$ denotes the vertex image of $S$ in the quotient.

The identity in Theorem~\ref{thm:w-main} can be written as follows
\begin{equation}\label{eq:det-symanzik}
	\det D[S] = (-1)^{|S|-1} 2^{|S|-2} \left( \Big(\sum_{e \in E}\alpha_e \Big)\, \psi_{(G/S)}(\underline{\alpha}) - \varphi_{(G/S)}(p_{\rm can}; \underline{\alpha}) \right) .
\end{equation}
The ``normalized'' principal minor \eqref{eq:det-cof} can be written using Symanzik polynomials as
\begin{equation}\label{eq:det-cof-symanzik}
	\frac{\det D[S]}{\cof D[S]} = \frac12 \left( \sum_{e \in E} \alpha_e - \frac{\varphi_{(G/S)}(p_{\rm can}; \underline{\alpha})}{\psi_{(G/S)}(\underline{\alpha})} \right) .
\end{equation}

\subsection{Related work}

The idea to study the ratio $({\det D}/{\cof D})$ appeared early on in the literature on distance matrices.
In \cite{graham-hoffman-hosoya}, it is shown that this ratio is additive on wedge sums of graphs.
This powerful observation is enough to essentially obtain the Graham--Pollak identity \eqref{eq:full-det} and its weighted version \eqref{eq:w-full-det} as corollaries.

Choudhury and Khare~\cite{choudhury-khare} prove a generalization of the Graham--Pollak identity \eqref{eq:full-det} in which $D$ can be replaced with certain principal minors, but crucially their generalization excludes all cases where $\det D[S]$ depends on the underlying tree structure.
Namely, they consider precisely those minors $\det D[S]$ which can be expressed in terms of data assigned to edges (weights, edge counts, etc.), without any reference to what edges are incident to others.
In contrast, Theorems~\ref{thm:main} and \ref{thm:w-main} make evident use of the tree structure, via the spanning forests in $\forests_1(G; S)$ and $\forests_2(G; S)$.

It is natural to ask how our results for trees may be generalized to arbitrary finite graphs. We address this in an upcoming paper \cite{richman-shokrieh-wu}, which involves more technical machinery.
In another recent paper, we use related potential theoretic ideas to prove bounds on the number of two-component spanning forests in a graph~\cite{richman-shokrieh-wu-2}.
Some ideas in the current paper overlap with those developed in~\cite{devriendt-thesis}.
Some generalizations of our results, namely Theorem~\ref{thm:monotonic}, appear already in \cite{devriendt-thesis}*{Property 3.38}, which is concerned with effective resistances on graphs.
On a tree, effective resistance is the same as (path-)distance.
If $D$ is the distance matrix of a tree, then any principal minor $D[S]$ is the effective resistance matrix of some (edge-weighted) graph on the vertex set $S$.
The quantity $(\det D[S] / \cof D[S])$ is referred to as the {\em resistance radius} of $S$ (up to a factor of $2$),
and the vector $\boldm$ used in \S\ref{sec:distance_proofs} is called the {\em resistance curvature} (up to scaling) in Devriendt's terminology.
The idea to consider an $\boldm$-like vector as a ``curvature'' also appears in \cite{steinerberger}, which motivates the terminology by proving graph-theoretic analogues of some curvature-related theorems in differential geometry. 
The results in \S\ref{sec:distance_prelim} regarding the signature of the matrices $D[S]$ can be considered special discrete cases of the {\em Dirichlet pairing} defined by
Baker and Rumely~\cite{baker-rumely}.

Amini studies ratios of Symanzik polynomials in \cite{amini}, similar to the ratio appearing in \eqref{eq:det-cof-symanzik}, motivated by quantum field theory and the calculation of Feynman amplitudes~\cite{amini-bloch-etal}.
Amini's results apply for Symanzik polynomials on an arbitrary graph, in contrast to \eqref{eq:det-cof-symanzik} where the relevant graph is a vertex-quotient of a tree.
The first Symanzik polynomial can be computed using Kirchhoff's celebrated matrix-tree theorem.
The second Symanzik polynomial can similarly be calculated via determinant of the generalized Laplacian matrix; see 
\cite[Section 1.1]{amini} and \cite[Theorem 7.1]{brown}.

In \cite{gutierrez-lillo}, Guti\'{e}rrez and Lillo observe that Theorem~\ref{thm:main} can be expressed as
\begin{multline}
	\det D[S] = (-1)^{|S| - 1} 2^{|S| - 2} \Big( (|S| - 1)\, \kappa(G;S)  \\
	- \!\!\! \sum_{F \in \forests_2(G;S)} \left(\outdeg(F,*) - 1\right)\left(\outdeg(F,*) - 4\right)  \Big)
\end{multline}
through some straightforward algebraic manipulation,
while also providing a new proof for this result using a nice combinatorial argument involving sign-reversing involutions on collections of paths in $G$.

The results in this paper may be of interest to those studying phylogenetics.
In phylogenetics, one aims to find the tree that best represents the evolutionary history among a collection of organisms using biological data (e.g. DNA sequences).
In this tree, leaf vertices represent modern-day species, while internal vertices represent ancestral species.
There are standard methods for estimating pairwise distances between species along their evolutionary tree. 
This means we can often predict the distance submatrix $D[S]$ of the target tree, in which $S$ is the set of leaf vertices, and we would like to use this information in reverse to decide what underlying tree best fits this distance data.
Results such as Theorem~\ref{thm:monotonic} may lead to new tests for phylogenetic inference, or for evaluating tree instability~\cite{collienne-etal}.

\subsection*{Structure of the paper}

In \S\ref{sec:graphs-matrices} we review terminology and notation concerning graphs, spanning forests, Laplacian matrices, and degree-related identities.
In \S\ref{sec:distance_prelim} we recall some results on the spectral properties of the distance matrix and its principal submatrices.
In \S\ref{sec:optimization} we show that the normalized principal minor $(\det D[S] / \cof D[S])$ is the solution to some quadratic optimization problem involving the bilinear form $D[S]$.
This observation is used to prove
Theorem~\ref{thm:monotonic} (monotonicity).
In \S\ref{sec:distance_proofs} we prove Theorems~\ref{thm:main} and \ref{thm:w-main}, which give a combinatorial expression for $\det D[S]$ in terms of rooted spanning forests.
In \S\ref{sec:char-poly} we prove Theorem~\ref{thm:dist-poly-coeff} on the coefficients of the distance matrix characteristic polynomial.
We end the paper with some examples in \S\ref{sec:examples} demonstrating Theorems~\ref{thm:main} and \ref{thm:w-main}.

\subsection*{Acknowledgements}
The authors would like to thank Ravindra Bapat for helpful correspondence,
in particular related to the results in \S~\ref{sec:signature}. 
We also thank Matt Baker, Karel Devriendt, Apoorva Khare, and Sebastian Prillo for helpful discussions.
We thank \'{A}lvaro Guti\'{e}rrez for helpful discussions regarding his follow-up to our work.

HR was supported by the Howard Hughes Medical Institute and the Matsen Group at the Fred Hutchinson Cancer Center. 
FS was partially supported by NSF CAREER DMS-2044564 grant.
CW was partially supported by Simons Collaboration Grant 850685.

\section{Graphs and spanning forests}
\label{sec:graphs-matrices}

Throughout, let $G = (V, E)$ be a finite graph with (positive, real) edge lengths $\{ \alpha_e \colon e \in E\}$. 
A graph without assignment of edge lengths will be considered as a special case, with $\alpha_e=1$ for all $e \in E$.

We assume all graphs in the paper are equipped with an implicit (arbitrary) orientation. This means we fix a pair of functions $\mathrm{head}\colon E \to V$ and $\mathrm{tail}\colon E \to V$, such that $\mathrm{head}(e)$ and $\mathrm{tail}(e)$ are the endpoints of $e$. We abbreviate $\mathrm{head}(e)$ as $e^+$, and $\mathrm{tail}(e)$ as $e^-$.


\subsection{Spanning trees and forests}

A {\em spanning tree} of a graph $G$ is a subgraph which 
is connected, has no cycles,
and contains all vertices of $G$.
A {\em spanning forest} of a graph $G$ is a subgraph which 
has no cycles
and  contains all vertices of $G$.

Given a nonempty set of vertices $S$  
an {\em $S$-rooted spanning forest} of $G$ 
is a spanning forest which has exactly one vertex of $S$ in each connected component.
Given $s \in S$ and such a forest $F$, we let $F(s)$ denote the {\em $s$-component} of $F$, that is, the component of $F$ containing $s$.

An {\em $(S,*)$-rooted spanning forest} of $G$ is a spanning forest which has $|S|+1$ components, where $|S|$ components each contain one vertex of $S$, and the additional component is disjoint from $S$.
We call the component disjoint from $S$ the {\em floating component}, following terminology in \cite{kassel-kenyon-wu}.
For an $(S,*)$-rooted spanning forest $F$, we let $F(*)$ denote the floating component. We sometimes refer to the floating component as the $*$-component of $F$. Again, for $s \in S$, we let $F(s)$ denote the $s$-component of $F$.

Let $\trees(G;S)$ denote the set of $S$-rooted spanning forests of $G$,
and let $\forests_2(G;S)$ denote the set of $(S,*)$-rooted spanning forests of $G$.

\begin{eg}\label{eg:running}
Suppose $G$ is the tree with unit edge lengths shown below.
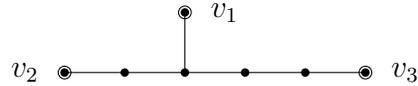
\[
\begin{tikzpicture}[scale=0.8]
	\coordinate (1) at (0,0);
	\coordinate (A) at (1,0);
	\coordinate (B) at (2,0);
	\coordinate (C) at (3,0);
	\coordinate (D) at (4,0);
	\coordinate (2) at (5,0);
	\coordinate (3) at (2,1);
	
	\foreach \c in {1,2,3,A,B,C,D} {
		\fill (\c) circle (2pt);
	}
	\foreach \c in {1,2,3} {
		\draw (\c) circle (3pt);
	}

	\draw (1) -- (A);
	\draw (A) -- (B);
	\draw (B) -- (3);
	\draw (B) -- (C);
	\draw (C) -- (D);
	\draw (D) -- (2);
	
	\foreach \c/\d/\label in {1/left/v_2,2/right/v_3,3/right/v_1} {
		\node[\d=0.2] at (\c) {$\label$};
	}
\end{tikzpicture}
\]
Let $S$ be the set of three leaf vertices.
Then $\trees(G;S)$ contains $11$ forests,
while $\forests_2(G;S)$ contains $19$ forests.
Some of these are shown in Figure~\ref{fig:1-forests} and Figure~\ref{fig:2-forests}, respectively.

\begin{figure}[h]
\centering
\begin{tikzpicture}[scale=0.5]
	\coordinate (1) at (0,0);
	\coordinate (A) at (1,0);
	\coordinate (B) at (2,0);
	\coordinate (C) at (3,0);
	\coordinate (D) at (4,0);
	\coordinate (2) at (5,0);
	\coordinate (3) at (2,1);
	
	\foreach \c in {1,2,3,A,B,C,D} {
		\fill (\c) circle (2pt);
	}
	\foreach \c in {1,2,3} {
		\draw (\c) circle (3pt);
	}

	\draw (1) -- (A);
	\draw[dotted] (A) -- (B);
	\draw (B) -- (C);
	\draw (C) -- (D);
	\draw (D) -- (2);
	\draw[dotted] (B) -- (3);
\end{tikzpicture}
\qquad
\begin{tikzpicture}[scale=0.5]
	\coordinate (1) at (0,0);
	\coordinate (A) at (1,0);
	\coordinate (B) at (2,0);
	\coordinate (C) at (3,0);
	\coordinate (D) at (4,0);
	\coordinate (2) at (5,0);
	\coordinate (3) at (2,1);
	
	\foreach \c in {1,2,3,A,B,C,D} {
		\fill (\c) circle (2pt);
	}
	\foreach \c in {1,2,3} {
		\draw (\c) circle (3pt);
	}

	\draw (1) -- (A);
	\draw[dotted] (A) -- (B);
	\draw (B) -- (C);
	\draw (C) -- (D);
	\draw[dotted] (D) -- (2);
	\draw (B) -- (3);
\end{tikzpicture}
\qquad
\begin{tikzpicture}[scale=0.5]
	\coordinate (1) at (0,0);
	\coordinate (A) at (1,0);
	\coordinate (B) at (2,0);
	\coordinate (C) at (3,0);
	\coordinate (D) at (4,0);
	\coordinate (2) at (5,0);
	\coordinate (3) at (2,1);
	
	\foreach \c in {1,2,3,A,B,C,D} {
		\fill (\c) circle (2pt);
	}
	\foreach \c in {1,2,3} {
		\draw (\c) circle (3pt);
	}

	\draw (1) -- (A);
	\draw[dotted] (A) -- (B);
	\draw (B) -- (C);
	\draw[dotted] (C) -- (D);
	\draw (D) -- (2);
	\draw (B) -- (3);
\end{tikzpicture}
\qquad
\begin{tikzpicture}[scale=0.5]
	\coordinate (1) at (0,0);
	\coordinate (A) at (1,0);
	\coordinate (B) at (2,0);
	\coordinate (C) at (3,0);
	\coordinate (D) at (4,0);
	\coordinate (2) at (5,0);
	\coordinate (3) at (2,1);
	
	\foreach \c in {1,2,3,A,B,C,D} {
		\fill (\c) circle (2pt);
	}
	\foreach \c in {1,2,3} {
		\draw (\c) circle (3pt);
	}

	\draw (1) -- (A);
	\draw[dotted] (A) -- (B);
	\draw[dotted] (B) -- (C);
	\draw (C) -- (D);
	\draw (D) -- (2);
	\draw (B) -- (3);
\end{tikzpicture}
\caption{Some forests in $\trees(G;S)$.}
\label{fig:1-forests}
\end{figure}

\begin{figure}[h]
\centering
\begin{tikzpicture}[scale=0.5]
	\coordinate (1) at (0,0);
	\coordinate (A) at (1,0);
	\coordinate (B) at (2,0);
	\coordinate (C) at (3,0);
	\coordinate (D) at (4,0);
	\coordinate (2) at (5,0);
	\coordinate (3) at (2,1);
	
	\foreach \c in {1,2,3,A,B,C,D} {
		\fill (\c) circle (2pt);
	}
	\foreach \c in {1,2,3} {
		\draw (\c) circle (3pt);
	}

	\draw (1) -- (A);
	\draw[dotted] (A) -- (B);
	\draw[dotted] (B) -- (3);
	\draw (B) -- (C);
	\draw (C) -- (D);
	\draw[dotted] (D) -- (2);

	\draw[color=red, line width=12pt, cap=round, opacity=0.1] (B) -- (D);
\end{tikzpicture}
\qquad
\begin{tikzpicture}[scale=0.5]
	\coordinate (1) at (0,0);
	\coordinate (A) at (1,0);
	\coordinate (B) at (2,0);
	\coordinate (C) at (3,0);
	\coordinate (D) at (4,0);
	\coordinate (2) at (5,0);
	\coordinate (3) at (2,1);
	
	\foreach \c in {1,2,3,A,B,C,D} {
		\fill (\c) circle (2pt);
	}
	\foreach \c in {1,2,3} {
		\draw (\c) circle (3pt);
	}

	\draw (1) -- (A);
	\draw[dotted] (A) -- (B);
	\draw[dotted] (B) -- (3);
	\draw (B) -- (C);
	\draw[dotted] (C) -- (D);
	\draw (D) -- (2);

	\draw[color=red, line width=12pt, cap=round, opacity=0.1] (B) -- (C);
\end{tikzpicture}
\qquad
\begin{tikzpicture}[scale=0.5]
	\coordinate (1) at (0,0);
	\coordinate (A) at (1,0);
	\coordinate (B) at (2,0);
	\coordinate (C) at (3,0);
	\coordinate (D) at (4,0);
	\coordinate (2) at (5,0);
	\coordinate (3) at (2,1);
	
	\foreach \c in {1,2,3,A,B,C,D} {
		\fill (\c) circle (2pt);
	}
	\foreach \c in {1,2,3} {
		\draw (\c) circle (3pt);
	}

	\draw (1) -- (A);
	\draw[dotted] (A) -- (B);
	\draw (B) -- (3);
	\draw (B) -- (C);
	\draw[dotted] (C) -- (D);
	\draw[dotted] (D) -- (2);

	\draw[color=red, line width=12pt, cap=round, opacity=0.1] (D) -- (D);
\end{tikzpicture}
\qquad
\begin{tikzpicture}[scale=0.5]
	\coordinate (1) at (0,0);
	\coordinate (A) at (1,0);
	\coordinate (B) at (2,0);
	\coordinate (C) at (3,0);
	\coordinate (D) at (4,0);
	\coordinate (2) at (5,0);
	\coordinate (3) at (2,1);
	
	\foreach \c in {1,2,3,A,B,C,D} {
		\fill (\c) circle (2pt);
	}
	\foreach \c in {1,2,3} {
		\draw (\c) circle (3pt);
	}

	\draw (1) -- (A);
	\draw[dotted] (A) -- (B);
	\draw (B) -- (3);
	\draw[dotted] (B) -- (C);
	\draw[dotted] (C) -- (D);
	\draw (D) -- (2);

	\draw[color=red, line width=12pt, cap=round, opacity=0.1] (C) -- (C);
\end{tikzpicture}
\caption{Some forests in $\forests_2(G;S)$, with floating component highlighted.}
\label{fig:2-forests}
\end{figure}
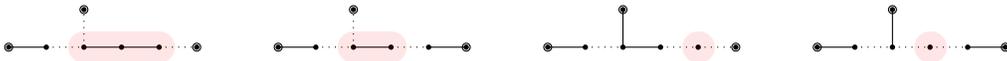
\end{eg}

\begin{rmk}
	Let $q$ be a fixed vertex. Then $\trees(G;\{q\})$ is the set of spanning trees of $G$, and $\forests_2(G;\{q\})$ is the set of two-component spanning forests of $G$. 
	Note that both these sets are independent of the choice of the vertex $q$, and were denoted by $\trees(G)$ and $\forests_2(G)$, respectively, in \S\ref{sec:symanzik}.
\end{rmk}

\subsection{Laplacian matrix}

The {\em incidence matrix} of a graph $G = (V,E)$ (endowed with an arbitrary orientation) is the matrix $B \in \RR^{V \times E}$ defined by
\[
	B_{v, e} = \one(v = e^+) - \one(v = e^-).
\]
Here $\one(\cdot)$ denotes the indicator function.

Let $\{ \alpha_e \colon e \in E\}$ denote the edge lengths as before. 
The Laplacian matrix is defined by
\begin{equation}\label{eq:weighted-laplacian}
	\La = B \begin{pmatrix}
	\alpha_1^{-1} & & \\
	& \ddots & \\
	& & \alpha_m^{-1}
	\end{pmatrix} B^\tr .
\end{equation}
It is clear that $\La$ is positive semidefinite, since edge lengths are positive. 
Note that the incidence matrix depends on the choice of the orientation but the Laplacian matrix does not.

\subsection{Principal minors matrix-tree theorem}

For a subset $A \subseteq E$ we define $w(A)$ as
\[
	w(A) = \prod_{e \not\in A} \alpha_e .
\]
If $H$ is a subgraph of $G$, we use $w(H)$ to denote $w(E(H))$. 
Let 
\[
	\kappa(G;S) = \sum_{T \in \trees(G; S)} w(T) .
\]
When all edge lengths are one, $\kappa(G; S)$ is the number of $S$-rooted spanning forests.

\begin{rmk}
If $G$ is a graph {\em without} assignment of edge lengths (so $\alpha_e=1$ for all $e \in E$), then $\kappa(G;S)$ is simply the number of $S$-rooted spanning forests of $G$.
In this case, $\kappa(G;S)$ is also the number of spanning trees of the quotient graph $G / S$, which ``glues together'' all vertices in $S$ as a single vertex $v_S$.
\end{rmk}

The {\em (principal minors) matrix-tree theorem} gives a determinantal formula for computing $\kappa(G;S)$, which we now explain. 
Given $S \subseteq V$, let $L[\overline S]$ denote the matrix obtained from $L$ by removing the rows and columns indexed by $S$.

\begin{thm}
\label{thm:matrix-tree}
Let $G = (V,E)$ be a finite graph with edge lengths $\{ \alpha_e \colon e \in E\}$.
For any nonempty vertex set $S \subseteq V$,
\begin{equation*}
	\kappa( G ; S) = \left(\prod_{e \in E} \alpha_e \right) \det L[\overline S] .
\end{equation*}
\end{thm}
\begin{proof}
See Tutte~\cite[Section VI.6, Equation (VI.6.7)]{tutte}.
\end{proof}
Note that the classical (Kirchhoff's) matrix-tree theorem is the special case where $S$ is a singleton.

\subsection{Tree splits and tree distance}
\label{sec:tree-splits}

Fix a tree $G = (V,E)$ with edge lengths $\{ \alpha_e \colon e \in E\}$ (and endowed with an arbitrary orientation). Given an edge $e \in E$, the edge deletion $G \setminus e$ contains two connected components. This phenomenon is referred to as a {\em tree split}.

Using the orientation on $e = (e^+,e^-)$,
we denote by $(G \setminus e)^\pm$ the component that contains endpoint $e^\pm$. Similarly, given $e\in E$ and a vertex $v$, we denote by $(G \setminus e)^{v}$ the component of $G\setminus e$ containing $v$, and $(G\setminus e)^{\overline v}$ the other.

Tree splits can be used to express the path distances between vertices in a tree as we explain next.
Given an edge $e\in E$ and vertices $v,w \in V$, let 
\begin{equation*}
\delta(e;v,w) = \begin{cases}
	1 &\text{if $e$ separates  $v$ from $w$}, \\
	0 &\text{otherwise}.
\end{cases}
\end{equation*}

\begin{rmk}\label{rmk:delta-properties}
\hfill
\begin{enumerate}[label=(\roman*)]
\item If we fix an edge $e$, then $\delta(e; -, -): V \times V \to \{0, 1\}$ is the indicator function for vertex pairs $(v, w)$ that are on opposite sides of the tree split $G \setminus e$.

\item 
In particular, if we fix $e$ and a vertex $v$,
then $\delta(e;v, -) \colon V(G) \to \{0,1\}$ 
is the indicator function for the component 
$(G \setminus e)^{\overline v}$ of the tree split $G \setminus e$ not containing $v$.

\item If we fix vertices $v, w$, then $\delta(-;v,w) \colon E(G) \to \{0,1\}$
is the indicator function for the unique path between $v$ and $w$.
\end{enumerate}
\end{rmk}

Recall the distance between vertices $v,w$ (with respect to edge lengths $\{ \alpha_e \colon e \in E\}$), denoted by $d(v,w)$, is the sum of the edge lengths $\alpha_e$ along the unique path from $u$ to $v$. Now Remark~\ref{rmk:delta-properties} (iii) implies the following.

\begin{prop}
\label{prop:distance-sum}
For a tree $G = (V,E)$ with edge lengths $\{\alpha_e \colon e \in E\}$, the distance function satisfies
\[ 
	d(v,w) = \sum_{e \in E} \alpha_e \, \delta(e; v,w) .
\]
\end{prop}

\subsection{Outdegree of forest components}
\label{sec:outdegree}

Given a vertex $v$ in a graph, the {\em degree} $\deg(v)$ is the number of edges incident to $v$.
A consequence of the ``handshaking lemma'' in graph theory is that for any tree $G$, we have
\begin{equation}\label{eq:handshake}
	\sum_{v \in V(G)} (2 - \deg(v)) = 2.
\end{equation}
We next state a slight generalization of \eqref{eq:handshake}, which will be useful later.

Given a connected subgraph $H \subseteq G$,
we define the {\em edge boundary} $\partial H$ as the set of edges which join $H$ to its complement; i.e.
\begin{equation*}
	\partial H = \{ e = \{a,b\} \in E \colon a \in V(H),\, b \not\in V(H)\}.
\end{equation*}
Define the {\em outdegree} of $H$ in $G$ as the number of edges in its edge boundary, $\outdeg(H) = |\partial H|$.
Note that the edge boundary and outdegree do not depend on the implicit orientation on $E$.

We are especially interested in the following special cases. 
Let $G$ be a tree and $\emptyset \ne S \subseteq V$. 
For an $S$-rooted spanning forest $F$ of $G$, and $s \in S$, define the {\em outdegree} $\outdeg(F,s)$ as the number of edges which join $F(s)$ (the $s$-component of $F$) to a different component.
Similarly, if $F$ is a $(S,*)$-rooted spanning forest of $G$, let $\outdeg(F,*)$ denote the outdegree of the floating component $F(*)$.
\begin{eg}\label{eg:outdegrees}
	Consider the tree $G$ from Example~\ref{eg:running}, where $S$ is the set of three leaf vertices.
	In Figure~\ref{fig:outdeg}, we show examples of rooted spanning forests of $G$, where each component is labelled by $\outdeg(F, s)$, for a leaf vertex $s$, or $\outdeg(F, *)$ for the floating component.

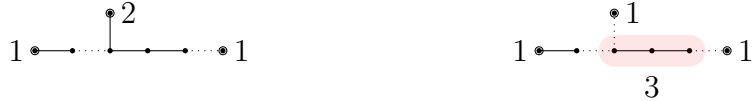
\begin{figure}[h]
	\begin{minipage}{0.45\textwidth}
		\centering
		\begin{tikzpicture}[scale=0.5]
			\coordinate (1) at (0,0);
			\coordinate (A) at (1,0);
			\coordinate (B) at (2,0);
			\coordinate (C) at (3,0);
			\coordinate (D) at (4,0);
			\coordinate (2) at (5,0);
			\coordinate (3) at (2,1);
			
			\foreach \c in {1,2,3,A,B,C,D} {
				\fill (\c) circle (2pt);
			}
			\foreach \c in {1,2,3} {
				\draw (\c) circle (3pt);
			}
		
			\draw (1) -- (A);
			\draw[dotted] (A) -- (B);
			\draw (B) -- (C);
			\draw (C) -- (D);
			\draw[dotted] (D) -- (2);
			\draw (B) -- (3);

			\node[left] at (1) {$1$};
			\node[right] at (3) {$2$};
			\node[right] at (2) {$1$};
			\node[below=0.2] at (C) {\phantom{$3$}}; 
		\end{tikzpicture}
	\end{minipage}
	\begin{minipage}{0.45\textwidth}
		\centering
		\begin{tikzpicture}[scale=0.5]
			\coordinate (1) at (0,0);
			\coordinate (A) at (1,0);
			\coordinate (B) at (2,0);
			\coordinate (C) at (3,0);
			\coordinate (D) at (4,0);
			\coordinate (2) at (5,0);
			\coordinate (3) at (2,1);
			
			\foreach \c in {1,2,3,A,B,C,D} {
				\fill (\c) circle (2pt);
			}
			\foreach \c in {1,2,3} {
				\draw (\c) circle (3pt);
			}
		
			\draw (1) -- (A);
			\draw[dotted] (A) -- (B);
			\draw[dotted] (B) -- (3);
			\draw (B) -- (C);
			\draw (C) -- (D);
			\draw[dotted] (D) -- (2);
		
			\draw[color=red, line width=12pt, cap=round, opacity=0.1] (B) -- (D);

			\node[left] at (1) {$1$};
			\node[right] at (3) {$1$};
			\node[right] at (2) {$1$};
			\node[below=0.2] at (C) {$3$};
		\end{tikzpicture}
	\end{minipage}
	\caption{An $S$-rooted spanning forest (left) and an $(S, *)$-rooted spanning forest (right).
	Each forest component is labelled with its outdegree.}
	\label{fig:outdeg}
\end{figure}
\noindent
	In consideration of the next lemma, we also show $G$ labelled with $2 - \deg(v)$ on each vertex, below.
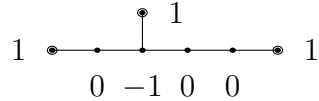
\[
\begin{tikzpicture}[xscale=0.6,yscale=0.5]
	\coordinate (1) at (0,0);
	\coordinate (A) at (1,0);
	\coordinate (B) at (2,0);
	\coordinate (C) at (3,0);
	\coordinate (D) at (4,0);
	\coordinate (2) at (5,0);
	\coordinate (3) at (2,1);
	
	\foreach \c in {1,2,3,A,B,C,D} {
		\fill (\c) circle (2pt);
	}
	\foreach \c in {1,2,3} {
		\draw (\c) circle (3pt);
	}

	\draw (1) -- (A);
	\draw (A) -- (B);
	\draw (B) -- (3);
	\draw (B) -- (C);
	\draw (C) -- (D);
	\draw (D) -- (2);
	
	\foreach \pt/\dir/\label in {
		1/left/1,2/right/1,3/right/1,A/below/0,B/below/-1,C/below/0,D/below/0
	} {
		\node[\dir=0.2] at (\pt) {$\label$};
	}
\end{tikzpicture}
\]
\end{eg}

\begin{lem}\label{lem:outdeg-sum}
Suppose $G$ is a tree.
\begin{enumerate}
	\item
	If $H \subseteq G$ is a nonempty connected subgraph then
	\[
		\sum_{v \in V(H)} \left( 2 -  \deg(v) \right) = 2 - \outdeg(H) .
	\]
	\item 
	For any fixed edge $e$ and fixed vertex $u$ of $G$, we have
	\[
		\sum_{v \in V(G)} (2 - \deg(v))\, \delta(e; u,v) = 1.
	\]
\end{enumerate}
\end{lem}
\begin{proof}\hspace{0pt}
\begin{enumerate}
	\item This is straightforward to check by induction on $|V(H)|$,
	with base case $|V(H)| = 1$:
	if $H = \{v\}$ consists of a single vertex, then $\outdeg(H) = \deg(v)$.
	\item Recall that $(G \setminus e)^{\overline u}$ denotes the component of the tree split $G \setminus e$ that does not contain $u$.
	Its vertices are precisely those $v$ that satisfy $\delta(e; u, v) = 1$ (see Remark~\ref{rmk:delta-properties} (iii)), so
	\[
		\sum_{v \in V} (2 - \deg(v)) \delta(e; u,v)
		= \sum_{v \in (G \setminus e)^{\overline u}} (2 - \deg(v)) .
	\]
	Since the subgraph $(G \setminus e)^{\overline u}$ has a single edge separating it from its complement, we have
	$\outdeg((G \setminus e)^{\overline u}) = 1$.
	Apply part (1) to obtain the result.
\end{enumerate}
\end{proof}

\subsection{Transitions between $\trees(G; S)$ and $\forests_2(G; S)$}
A key step in the proof of Theorem~\ref{thm:w-main} uses the following ``transition structure'' which relates the $S$-rooted spanning forests $\trees(G; S)$ with $(S, *)$-rooted spanning forests $\forests_2(G; S)$,
via the operations of edge-deletion and edge-union.

\begin{itemize}[leftmargin=*]
\item Consider the ``deletion map''
\[
	E(G) \times \trees(G;S) \xrightarrow{\mathrm{del}} \trees(G;S) \sqcup \forests_2(G;S)
\]
defined by
\[
	(e, T) \mapsto \begin{cases}
	T &\text{if } e\not\in T,\\
	T \setminus e &\text{if } e\in T.
	\end{cases}
\]
Given $F \in \forests_2(G;S)$,
there are exactly $\outdeg(F, *)$-many choices of pairs $(e, T) \in E(G) \times \trees(G;S)$ such that $F = T \setminus e$.
\medskip
\item Consider the ``union map'' 
\[
	E(G) \times \forests_2(G;S) \longrightarrow \trees(G;S) \sqcup \forests_2(G;S)
\]
defined by
\[
	(e, F) \mapsto \begin{cases}
		F \cup e &\text{if } e \in \partial F(*), \\
		F &\text{if } e \not\in \partial F(*)
	\end{cases}
\]
Given $T \in \trees(G; S)$,  
there are exactly $(|V| - |S|)$-many choices of pairs $(e, F) \in E(G) \times \forests_2(G; S)$ such that $T = F \cup e$.
\end{itemize}

\section{Principal submatrices of \texorpdfstring{$D$}{D} as bilinear forms}
\label{sec:distance_prelim}

In the rest of this paper, let $G = (V, E)$ be a tree with edge lengths $\alpha = \{\alpha_e \colon e\in E\}$. The {distance matrix} (with respect to $\alpha$) is symmetric, so it defines a symmetric bilinear forms on the vector space $\RR^V$.

\subsection{Signature and invertibility}\label{sec:signature}

For a subset of vertices $S$ with $|S| \geq 2$, the submatrix $D[S]$ has nonzero determinant. 
We give a proof in here, based on finding the signature of $D[S]$ as a bilinear form.

We first recall a celebrated result of Cauchy.
Let $M[\overline{i}]$ denote the matrix obtained from $M$ by deleting the $i$-th row and column.

\begin{prop}[Cauchy interlacing]
\label{prop:cauchy-interlacing}
Suppose $M$ is a symmetric real matrix  with eigenvalues $\lambda_1 \leq \cdots \leq \lambda_n$. 
Then the submatrix $M[\overline{i}]$ has eigenvalues $\mu_1 \leq \cdots \leq \mu_{n-1}$ satisfying
\[
	\lambda_1 \leq \mu_1 \leq \lambda_2 \leq \cdots \leq \mu_{n-1} \leq \lambda_n.
\]
\end{prop}
\begin{proof}
See, for example, \cite[Theorem 4.3.17]{horn-johnson}.
\end{proof}

\begin{lem}[{\cite[Lemma 8.15]{bapat-book}}]
\label{lem:dist-signature}
Suppose $\Da$ is the distance matrix of a tree with $n$ vertices and edge lengths $\{\alpha_e \colon e \in E\}$. 
Then $\Da$ has one positive eigenvalue and $n - 1$ negative eigenvalues.
\end{lem}
\begin{rmk}
The proof in \cite{bapat-book} is by induction on the number of vertices, and uses Cauchy interlacing (Proposition~\ref{prop:cauchy-interlacing}). Lemma 8.15 in \cite{bapat-book} is actually stated for trees with all edge lengths $1$. However, the same argument applies to trees with edge lengths if one applies Bapat, Kirkland, Neumann's result \cite[Corollary 2.5]{bapat-kirkland-neumann} (see \eqref{eq:w-full-det}).
\end{rmk}
The following extension of Lemma~\ref{lem:dist-signature} and its proof was communicated to the authors by Bapat.
\begin{lem}
\label{lem:distance-sub-nonsingular}
Suppose $\Da$ is the distance matrix of a tree $G = (V,E)$ with edge lengths $\{\alpha_e \colon e \in E\}$. Let $S \subseteq V$ be a subset of size $|S| \geq 2$. 
Then $\Da[S]$ has one positive eigenvalue and $|S| - 1$ negative eigenvalues.
In particular, $\det \Da[S] \neq 0$.
\end{lem}
\begin{proof}
We apply decreasing induction on the size of $S$. 
For $S = V$, we have Lemma~\ref{lem:dist-signature}.
Suppose $|S| = k$, where $2 \leq k < n$, and assume the claim holds for all vertex subsets of size greater than $k$.
Let $S^+ \subseteq V$ be a set of $k + 1$ vertices containing $S$.
By induction hypothesis $D[S^+]$ has $k$ negative eigenvalues and one positive eigenvalue.
By Cauchy interlacing (Proposition~\ref{prop:cauchy-interlacing}), $D[S]$ must have at least $k - 1$ negative eigenvalues. 
Since all diagonal entries of $D[S]$ are zero, $D[S]$ has zero trace.  Thus the remaining eigenvalue of $D[S]$ must be positive.
\end{proof}

\subsection{Negative definite hyperplane}

We next prove that any principal submatrix of $D$ induces a negative semidefinite quadratic form on the hyperplane of zero-sum vectors. In other words, $D$ is a {\em conditionally} negative semidefinite matrix.

Bapat, Kirkland, and Neumann (\cite[Theorem 2.1]{bapat-kirkland-neumann}) prove that
\begin{equation}\label{eq:w-distance-inverse}
	(\Da)^{-1} \;=\; - \frac12 \La + \frac12 \Big( \sum_{e \in E} \alpha_e\Big)^{-1} \boldm\, \boldm^\tr
\end{equation}
where $\boldm$ is the vector with components $\boldm_v = 2 - \deg v$.
The special case of \eqref{eq:w-distance-inverse} for trees with all edge lengths $1$ had appeared in an earlier work by Graham and Lov\'asz (\cite[Lemma 1]{graham-lovasz}).

\begin{prop}
\label{prop:dist-laplacian}
Let $D$ denote the distance matrix of a tree with edge lengths $\{\alpha_e \colon e \in E\}$. Let $\La$ be the Laplacian matrix. 
Then
\[
	\Da \;=\; - \frac{1}{2} \Da \La \Da + \frac{1}{2} \left( \sum_{e \in E} \alpha_e\right) \bone \bone^\tr .
\]
\end{prop}
\begin{proof}
Multiply \eqref{eq:w-distance-inverse} by the all-ones vector $\bone$; since $\La \bone = 0$ and $\boldm^\tr \bone = 2$ (see \eqref{eq:handshake}), we obtain 
\[
	(\Da)^{-1} \bone \;=\; \Big( \sum_{e \in E} \alpha_e\Big)^{-1} \boldm .
\]
Therefore, 
\begin{equation}\label{eq:Dm}
\Da \boldm = \left( \sum_{e \in E} \alpha_e \right) \bone .
\end{equation}
The result now follows from multiplying \eqref{eq:w-distance-inverse} by $\Da$ on both sides, and using \eqref{eq:Dm}. 
\end{proof}

\begin{prop}
\label{prop:negdef-hyperplane}
Suppose $D$ is a distance matrix as above. 
Let $S \subset V$ be a nonempty vertex subset.
Let $H_{c, S} \subset \RR^S$ be the affine hyperplane of vectors whose coordinates sum to $c$.
Then the function $\boldu \mapsto \boldu^\tr D[S] \boldu$ is concave on $H_{c, S}$.
\end{prop}
\begin{proof}
Suppose $\boldu$ is in $H_{c, S}$, so that $\bone^\tr \boldu = c$.
Let $\boldw \in \RR^V$ be image of $\boldu$ under the natural inclusion $\RR^S \to \RR^V$.
Using Proposition~\ref{prop:dist-laplacian}, we obtain
\[
	\boldu^\tr D[S]\boldu = \boldw^\tr D \boldw = - \frac12 (D \boldw)^\tr L (D \boldw) + \frac12 \left(\sum \alpha_e\right) c^2.
\]
It is well-known (and readily follows from \eqref{eq:weighted-laplacian}) that the Laplacian matrix $L$ is positive semidefinite. Note that as $\boldu$ varies in $H_{c, S}$, the term $\frac12 \left(\sum \alpha_e\right) c^2$ remains constant.
\end{proof}

\section{Quadratic optimization}
\label{sec:optimization}

We explain how the quantity $({\det D[S]}/{\cof D[S]})$, introduced in \S\ref{sec:Normalized principal minors}, arises as the solution of a quadratic optimization problem.
This is used to prove Theorem~\ref{thm:monotonic}.
The material in this section is not used in the proofs of Theorems~\ref{thm:main} and \ref{thm:w-main}.

\begin{prop}
\label{prop:optimization}
Suppose $D$ is the distance matrix of a tree with edge lengths $\{\alpha_e \colon e \in E\}$. 
If $D[S]$ is the principal submatrix of a distance matrix indexed by a nonempty subset $S$ of vertices, then
\[
	\frac{\det D[S]}{\cof D[S]} = \max \{\boldu^\tr D[S] \boldu \colon \boldu \in \RR^S,\, \bone^\tr \boldu = 1 \}.
\]
\end{prop}

\begin{proof}
If $|S| = 1$ then $D[S]$ is the zero matrix and the statement is true trivially.

Assume $|S| \geq 2$.
Proposition~\ref{prop:negdef-hyperplane} implies that 
the objective function $\boldu \mapsto \boldu^\tr D[S]\boldu$ is concave on the domain $\bone^\tr \boldu = 1$, so any critical point is a global maximum.
To find the critical points using the method of Lagrange multipliers, we compute the gradients of the objective function and the constraint:
\[\nabla(\boldu^\tr D[S] \boldu)=2 D[S] \boldu, \nabla(\bone^\tr \boldu)=\bone\]
Then the optimal solution $\boldu^*$ is a vector satisfying
\[
	D[S] \boldu^* = \lambda \bone \qquad\text{for some }\lambda \in \RR.
\]
The constant $\lambda$ is, in fact, the optimal objective value, since
\[
	(\boldu^*)^\tr D[S] \boldu^* = (\boldu^*)^\tr (\lambda \bone) = \lambda (\bone^\tr \boldu^*)^\tr = \lambda.
\]
The last equality uses the given constraint $\bone^\tr \boldu = 1$.

On the other hand,
 $D[S]$ is invertible by Lemma~\ref{lem:distance-sub-nonsingular}. Therefore, we have $ \boldu^* = \lambda (D[S]^{-1} \bone) $, so
\[
	1 = \bone^\tr \boldu^* = \lambda (\bone^\tr D[S]^{-1} \bone)
	= \lambda \frac{\cof D[S]}{\det D[S]}.
\]
Thus the optimal objective value is
$\displaystyle
	\lambda = \frac{\det D[S]}{\cof D[S]} .
$
\end{proof}

\begin{rmk}
	If we consider $G$ as a resistive electrical network, with each edge $e$ representing a resistor of resistance $\alpha_e$,
	then the optimal vector $\boldu^*$ has a physical interpretation as a current flow: 
	it relates to the currents exiting from the nodes in $S$ when external current enters the network in the amount $(\deg(v) - 2)/2$ for each $v \in V \backslash S$, 
	and the network is grounded at all nodes in $S$.
	To be more precise, the entry of $\boldu^*$ at some $s \in S$ is equal to $(2 - \deg(s)) / 2$, minus the current described above.

	Indeed, we give an explicit combinatorial expression for $\boldu^*$, up to a normalizing constant, in \eqref{eq:m-vector}. 
	It is a classical result in network theory that this relates to the current flow as described; 
	see, for example, Tutte~\cite[Section VI.6]{tutte}.

	The vector in $\RR^V$ whose $v$-entry is $(2 - \deg(v))/2$ is called the {\em canonical measure} of the underlying tree by Chinburg--Rumely; see~\cite[Section 14]{baker-rumely}.
	We may consider $\boldu^*$ as an analogue of the {canonical measure} for the vertex subset $S$.
\end{rmk}

We note the following restatement of Proposition~\ref{prop:optimization}, viewing $\RR^S$ as a subspace of $\RR^V$ 
where coordinates indexed by $V \setminus S$ are set to zero.
\begin{cor}
\label{cor:optimization}
If $D$ is the distance matrix of a tree and  $S \subset V$, then
\[
	\frac{\det D[S]}{\cof D[S]} = \max \{\boldu^\tr D \boldu \colon \boldu \in \RR^V,\, \bone^\tr \boldu = 1,\, \boldu_v = 0 \text{ if } v \not\in S \}.
\]
\end{cor}

\subsection{Cofactor sums}

Next we recall a connection between minors of the Laplacian matrix and cofactor sums of the distance matrix, when $G$ is a tree.
The result is essentially due to Bapat and Sivasubramanian (\cite{bapat-sivasubramanian}).

Recall from \S\ref{sec:Normalized principal minors} that $\cof M$ for a square matrix $M$ denotes the {\em sum of cofactors} of $M$, i.e.
$\displaystyle
	\cof M = \sum_{i = 1}^{n} \sum_{j = 1}^{n} (-1)^{i + j} \det M[\overline{i}, \overline{j}]
$
where $M[\overline{i}, \overline{j}]$ denotes the matrix with the $i$-th row and $j$-th column deleted.
Recall that $\kappa(G; S)  :=  \sum_{T \in \trees(G; S)} w(T)$.

\begin{thm}[Distance submatrix cofactor sums]
\label{thm:distance-sub-cof}
Let $G = (V, E)$ be a tree with edge lengths $\{\alpha_e \colon e \in E\}$,
and let $\Da$ be the distance matrix of $G$.
Let $S \subseteq V$ be a nonempty subset of vertices. 
Then
\begin{equation*}
	\cof \Da[S] = (-2)^{|S| - 1} \kappa(G; S) .
\end{equation*}
\end{thm}
\begin{proof}
Bapat and Sivasubramanian (\cite[Theorem 11]{bapat-sivasubramanian})
show that
\[
	\cof \Da[S] = (-2)^{|S|-1} \left( \prod_{e \in E} \alpha_e \right) \det \La[\overline S] 
\]
where $\La$ is the Laplacian matrix.
Combine this identity with Theorem~\ref{thm:matrix-tree}.
\end{proof}

The following result is a direct consequence of theorems of Bapat, Kirkland, Neumann (\cite{bapat-kirkland-neumann}) and Bapat, Sivasubramanian (\cite{bapat-sivasubramanian}).

\begin{prop}\label{prop:full-det-cof-ratio}
	Suppose $\Da$ is the distance matrix of a tree with edge lengths $\{\alpha_e \colon e \in E\}$.
	Then
	\[
		\frac{\det \Da}{\cof \Da} = \frac1{2} \sum_{e \in E} \alpha_e .
	\]
\end{prop}
\begin{proof}
Consider Theorem~\ref{thm:distance-sub-cof} with $S = V$.
In this case $\trees(G; V)$ is a singleton consisting of the forest $T$ with no edges, so $w({T})$ is the product of all edge lengths.
Thus $\cof \Da = (-2)^{n - 1} \prod_{e \in E} \alpha_e$.
Combining this with \eqref{eq:w-full-det} yields the result.
\end{proof}

\subsection{Monotonicity}


Using Proposition~\ref{prop:optimization}, we can now prove Theorem~\ref{thm:monotonic}.

\begin{proof}[Proof of Theorem~\ref{thm:monotonic}]
By Corollary~\ref{cor:optimization}, both values 
$\displaystyle
	\frac{\det D[A]}{\cof D[A]}
$ 
and
$\displaystyle
	\frac{\det D[B]}{\cof D[B]}
$
arise from optimizing the same objective function on an affine subspace of $\RR^V$, 
and the subspace for $A$ is contained in the subspace for $B$.
\end{proof}

We get refined bounds by making use of the monotonicity property, Theorem~\ref{thm:monotonic}. 
Let $\conv(S,G)$ denote the convex hull of $S$, i.e. the subtree of $G$ consisting of all paths between points of $S$.
This is also known as the {\em Steiner tree} of the subset $S$.
 
\begin{thm}[Bounds on normalized principal minors]
\label{thm:det-cof-bounds}
	In the setting of Theorem~\ref{thm:monotonic}, we have the following.
	\begin{enumerate}[label=(\alph*)]
	\item
	For a nonempty subset $S \subseteq V$,
	\begin{equation*}
		\frac{\det \Da[S]}{\cof \Da[S]} \;\leq\; \frac12 \sum_{e \in E(\conv(S, G))} \alpha_e .
	\end{equation*}
	\item 
	For any vertices $s_0, s_1 \in S$,
	if $\gamma$ is the path between $s_0$ and $s_1$, then
	\begin{equation*}
		\frac12 \sum_{e \in \gamma} \alpha_e \;\leq\; \frac{\det \Da[S]}{\cof \Da[S]}.
	\end{equation*}
	\end{enumerate}
\end{thm}

\begin{proof}
\begin{enumerate}[label=(\alph*)]
\item Take $B$ to be the set of all vertices in $\conv(S, G)$.
Then $S \subseteq B$, so by Theorem~\ref{thm:monotonic}
and Proposition~\ref{prop:full-det-cof-ratio}, we have
\[
	\frac{\det D[S]}{\cof D[S]} \leq \frac{\det D[B]}{\cof D[B]} = \frac12 \sum_{e \in E(\conv(S, G))} \alpha_e .
\]
\item 
Let $A=\{s_0, s_1\}$.
Then $A \subseteq S$ by assumption. Now apply Theorem~\ref{thm:monotonic}
and Proposition~\ref{prop:full-det-cof-ratio}, we have
\[\frac{\det D[S]}{\cof D[S]}\geq
	\frac{\det D[A]}{\cof D[A]}
	= \frac12 d(s_0, s_1) 
	= \frac12 \sum_{e \in \gamma} \alpha_e .
	\qedhere
\]
\end{enumerate}
\end{proof}

\section{Determinant identities}
\label{sec:distance_proofs}

In this section, we prove our main result, Theorem~\ref{thm:w-main}, on the determinant of the principal submatrix $D[S]$.
Theorem~\ref{thm:main} is an immediate corollary.
The arguments do not make use of results in \S\ref{sec:optimization}.

\subsection{Outline of the proof}
\label{sec:proof_outline}

We compute $\det D[S]$ via the following steps.

\begin{enumerate}[label=(\Roman*)]
\item
Find the vector $\boldm \in \RR^S$ such that $D[S]\boldm = \lambda \bone$.
Find $\lambda$.

\item 
Compute the sum of entries of $\boldm$, i.e. $\bone^\tr \boldm$.

\item 
Since $D[S]$ is nonsingular (Lemma~\ref{lem:dist-signature}), $\boldm = \lambda (D[S]^{-1} \bone)$ and
\[
	\bone^\tr \boldm = \lambda (\bone^\tr D[S]^{-1} \bone) = \lambda \frac{\cof D[S]}{\det D[S]}.
\]
Calculate $\displaystyle \frac{\det D[S]}{\cof D[S]} = \frac{\lambda}{\bone^\tr \boldm}$.

\item
Multiply the previous expression by $\cof D[S]$, using the expression in Theorem~\ref{thm:distance-sub-cof}, to compute $\det D[S]$.
\end{enumerate}

It turns out that the entries of $\boldm$ are combinatorially meaningful (see \eqref{eq:m-vector}),
which also gives combinatorial meaning to
the constant $\lambda$.

\subsection{Equilibrium vector identities}

Fix a tree $G = (V,E)$ with edge lengths $\{\alpha_e \colon e \in E\}$ and a nonempty subset $S \subseteq V$.
We first define a vector $\boldm$ which satisfies the relation $D[S] \boldm = \lambda \bone$ for some $\lambda$.

\begin{dfn}
\label{dfn:m-vector}
Let $\boldm = \boldm(G;S)$ denote the vector in $\RR^S$ be defined by
\begin{equation}\label{eq:m-vector}
\boldm_v =  \sum_{T \in \trees(G;S)} w({T}) (2 - \outdeg(T,v))
\qquad\text{for each }v \in S.
\end{equation}
where $\outdeg(T,v)$ is the outdegree of the $v$-component of $T$ (see \S\ref{sec:outdegree}).
We call $\boldm$ the {\em equilibrium vector} of $(G; S)$.
\end{dfn}

The choice of terminology ``equilibrium'' here is in reference to potential theory, cf. \cite{Tsu, steinerberger}.

\begin{prop}
\label{prop:m-sum}
Suppose $S$ is nonempty.
For the vector $\boldm = \boldm(G; S)$ defined above, we have
\begin{enumerate}[label=(\alph*)]
\item
$\displaystyle \bone^\tr \boldm = 2 \,\sum_{T \in \trees(G;S)} w({T})$;

\item 
$\boldm$ is nonzero.
\end{enumerate}
\end{prop}
\begin{proof}
(a)
By Lemma~\ref{lem:outdeg-sum} we can express $\outdeg(T, s)$ as a sum over vertices in $T(s)$,
\[
	\boldm_s = \sum_{T \in \trees(G;S)} w({T}) (2 - \outdeg (T,s))
= \sum_{T \in \trees(G;S)} w({T}) \left( \sum_{v \in T(s)}(2 - \deg(v))\right).
\]
Then exchange the order of summation in $\bone^\tr \boldm$,
\begin{align*}
	\bone^\tr \boldm = \sum_{s\in S} \boldm_s &= \sum_{s \in S} \left( \sum_{T \in \trees(G;S)} w({T}) \sum_{v \in T(s)}(2 - \deg(v)) \right) \\
	&= \sum_{T \in \trees(G;S)} w({T}) \left( \sum_{s\in S} \sum_{v \in T(s)} (2 - \deg(v)) \right) .
\end{align*}
Observe that the inner double sum is simply a sum over $v \in V$,
since the vertex sets $\{T(s) \text{ for } s \in S\}$ form a partition of $V$ by definition of $T$ being an $S$-rooted spanning forest.
Thus, using  equation \eqref{eq:handshake}, we obtain:
\[
	\bone^\tr \boldm = \sum_{T \in \trees(F;S)} w({T}) \left( \sum_{v \in V} (2 - \deg(v))\right)
	= \sum_{T \in \trees(F;S)} w({T}) \cdot 2 
\]

(b) 
All edge lengths are positive, so $w({T}) > 0$ for all $T$, and $\trees(G; S)$ is nonempty as long as $S$ is nonempty. 
Therefore, part (a) implies that $\bone^\tr \boldm > 0$.
\end{proof}

The following computation is the technical heart of our main result.

\begin{thm}
\label{thm:m-distance-product}
With $\boldm = \boldm(G;S)$ defined as in \eqref{eq:m-vector},
$D[S] \boldm = \lambda \bone$
for the constant
\begin{equation}\label{eq:lambda}
	\lambda = \sum_{e \in E(G)} \alpha_e \sum_{T\in \trees(G;S)} w({T}) - \sum_{F \in \forests_2(G;S)} w({F}) \, (2 - \outdeg(F,*))^2.
\end{equation}
\end{thm}

\begin{proof}
For $e\in E$ and $v,w\in V$, let $\delta(e;v,w)$ denote the function defined in \S\ref{sec:tree-splits}.
For any $v \in S$, we have
\begin{align}
  (D[S] \boldm)_v &= \sum_{s \in S} d(v,s) \boldm_s \nonumber \\
  &= \sum_{s \in S} \left( \sum_{e \in E(G)} \alpha_e\, \delta(e; v,s) \right) \left( \sum_{T \in \trees(G;S)} (2 - \outdeg(T,s)) w({T}) \right) \nonumber \\
  &= \sum_{T\in \trees(G;S)} w({T}) \sum_{e\in E} \alpha_e \left( \sum_{s \in S} \delta(e; v,s) (2 - \outdeg(T, s)) \right) \nonumber \\
  &= \sum_{T \in \trees(G;S)} w({T}) \sum_{e \in E} \alpha_e \left( \sum_{s \in S} \delta(e; v,s) \sum_{u \in T(s)} (2 - \deg(u)) \right). \label{eq:14-1}
\end{align}
The second line follows from Proposition~\ref{prop:distance-sum} and \eqref{eq:m-vector}. 
The last line follows from using Lemma~\ref{lem:outdeg-sum} for the subgraph $H = T(s)$.

We introduce additional notation to handle the double sum in parentheses in \eqref{eq:14-1}.
Each $S$-rooted spanning forest $T$ naturally induces a surjection $\pi_T\colon V \to S$, defined by 
\[
	\pi_T(u) = s \qquad\text{if and only if}\qquad u \in T(s), \text{ the $s$-component of $T$}.
\]
We similarly extend $\pi_T$ to edges of $T$, by
\[
	\pi_T(e) = s \qquad\text{if and only if} \qquad e \in T(s),
\]
where we use the same symbol $\pi_T : E(T) \to S$ by some abuse of notation.

Using this notation,
\begin{equation}\label{eq:9}
	(D[S] \boldm)_v = \sum_{T \in \trees(G;S)} w({T}) \sum_{e \in E} \alpha_e \left( \sum_{u \in V} (2 - \deg(u)) \delta(e; v,\pi_T(u)) \right)
\end{equation}
We will compare how the above expression changes after replacing $\delta(e; v,\pi_T(u))$ with $\delta(e; v, u)$.

From Lemma~\ref{lem:outdeg-sum} (b), 
we have
$\displaystyle
	\sum_{u \in V} (2 - \deg(u)) \delta(e; v,u)
	= 1. 
$
Thus
\begin{equation}\label{eq:10}
	\sum_{T \in \trees(G;S)} w({T}) \sum_{e \in E} \alpha_e
	= \sum_{T \in \trees(G;S)} w({T}) \sum_{e \in E} \alpha_e \left( \sum_{u \in V} (2 - \deg(u)) \delta(e; v,u) \right) 
\end{equation}
By subtracting equation \eqref{eq:10} from \eqref{eq:9}, we obtain
\begin{multline}\label{eq:subtract-delta-diff}
	(D[S] \boldm)_v - \sum_{T \in \trees(G;S)} w({T}) \sum_{e \in E} \alpha_e \\
	= \sum_{T \in \trees(G;S)} w({T}) \sum_{e \in E} \alpha_e \sum_{u \in V} (2 - \deg(u)) \Big(\delta(e; v, \pi_T(u)) - \delta(e; v, u) \Big).
\end{multline}
For the rest of the argument, let
$\displaystyle
	(\star) = (D[S] \boldm)_v - \sum_{T \in \trees(G;S)} w({T}) \sum_{e \in E} \alpha_e.
$

We focus on the inner expression $\delta(e; v, \pi_T(u)) - \delta(e; v, u)$.
Recall that
\begin{equation}\label{eq:delta-diff}
	\delta(e; v, \pi_T(u)) - \delta(e; v, u) 
	= \one(\pi_T(u) \in (G \setminus e)^{\overline v}) - \one( u \in (G \setminus e)^{\overline v}) ,
\end{equation}
cf. Remark~\ref{rmk:delta-properties} (ii).
Now, consider varying $u$ over all vertices, when $e$, $T$, and $v$ are fixed.
There are three cases:

\begin{itemize}[leftmargin=*]
\item Case 1: if $e \not \in T$, then $u$ and $\pi_T(u)$ are on the same side of the tree split $G \setminus e$, for every vertex $u$.
In this case $\delta(e;v, \pi_T(\cdot)) - \delta(e; v, \cdot) = 0$.

\item Case 2: if $e \in T$ and $e$ separates $\pi_T(e)$ from $v$, then $\delta(e;v, \pi_T(\cdot)) - \delta(e; v, \cdot)$ is the indicator function for the floating component of $T \setminus e$. See Figure~\ref{fig:e-delete-from-forest}, left.

\item Case 3: if $e \in T$ and $e$ does not separate $\pi_T(e)$ from $v$, then $\delta(e;v, \pi_T(\cdot)) - \delta(e; v, \cdot)$ is the negative of the indicator function for the floating component of $T \setminus e$. See Figure~\ref{fig:e-delete-from-forest}, right.
\end{itemize}
\begin{figure}[h]
\begin{minipage}{0.45\textwidth}
\centering
\begin{tikzpicture}[scale=0.6]
	\coordinate (1) at (0,0);
	\coordinate (A) at (1,0);
	\coordinate (B) at (2,0);
	\coordinate (C) at (3,0);
	\coordinate (D) at (4,0);
	\coordinate (2) at (5,0);
	\coordinate (3) at (2,1);
	
	\foreach \c in {1,2,3,A,B,C,D} {
		\fill (\c) circle (2pt);
	}
	\foreach \c in {1,2,3} {
		\draw (\c) circle (3pt);
	}

	\draw (1) -- (A);
	\draw (A) -- node[below] {$e$} (B);
	\draw (B) -- (C);
	\draw[dotted] (C) -- (D);
	\draw (D) -- (2);
	\draw[dotted] (B) -- (3);

	\node[above=0.1] at (2) {$v$};
	\node[left] at (1) {$\pi_T(e)$};

	\draw[color=red, line width=12pt, cap=round, opacity=0.1] (B) -- (C);
\end{tikzpicture}
\end{minipage}
\begin{minipage}{0.45\textwidth}
\centering
\begin{tikzpicture}[scale=0.6]
	\coordinate (1) at (0,0);
	\coordinate (A) at (1,0);
	\coordinate (B) at (2,0);
	\coordinate (C) at (3,0);
	\coordinate (D) at (4,0);
	\coordinate (2) at (5,0);
	\coordinate (3) at (2,1);
	
	\foreach \c in {1,2,3,A,B,C,D} {
		\fill (\c) circle (2pt);
	}
	\foreach \c in {1,2,3} {
		\draw (\c) circle (3pt);
	}

	\draw (1) -- (A);
	\draw (A) -- node[below] {$e$} (B);
	\draw (B) -- (C);
	\draw[dotted] (C) -- (D);
	\draw (D) -- (2);
	\draw[dotted] (B) -- (3);

	\node[above left] at (1) {$\pi_T(e)$};
	\node[below left] at (1) {$v$};

	\draw[color=red, line width=12pt, cap=round, opacity=0.1] (B) -- (C);
\end{tikzpicture}
\end{minipage}
\caption{Edge $e \in T$ with $\delta(e; v, \pi_T(e)) = 1$ (left) and $\delta(e; v, \pi_T(e)) = 0$ (right). The floating component of $T \setminus e$ is highlighted.}
\label{fig:e-delete-from-forest}
\end{figure}
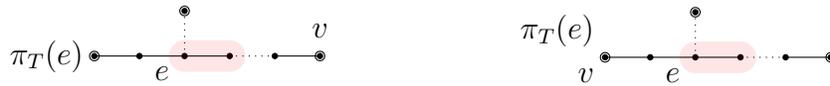

Thus, upon multiplying the term \eqref{eq:delta-diff} by $(2 - \deg(u))$ and summing over all vertices $u$, and applying Lemma~\ref{lem:outdeg-sum}, we obtain
\begin{align}
	\sum_{u \in V} (2 - \deg(u)) &\left(\delta(e; v, \pi_T(u)) - \delta(e; v, u)\right) \nonumber \\
	&= \begin{cases}
	0 &\text{if } e \not \in T, \\
	2 - \outdeg(T \setminus e, *) &\text{if } e \in T \text{ and } \delta(e; v, \pi_T(e)) = 1, \\
	-(2 - \outdeg(T \setminus e, *)) &\text{if } e \in T \text{ and } \delta(e; v, \pi_T(e)) = 0 .
	\end{cases} \label{delta-diff-sum}
\end{align}
Substituting \eqref{delta-diff-sum} into \eqref{eq:subtract-delta-diff} yields
\begin{multline}\label{eq:45}
	(\star)
	= \sum_{T\in \trees(G;S)} \! w({T}) \sum_{e \in T} \alpha_e \times \\
	( 2 - \outdeg(T\setminus e,*)) \Big( \one(\delta(e; v, \pi_T(e)) = 1) \! - \one(\delta(e; v, \pi_T(e)) = 0) \Big),
\end{multline}
where we use the shorthand
$\displaystyle
	(\star) = (D[S] \boldm)_v - \sum_{T \in \trees(G;S)} w({T}) \sum_{e \in E} \alpha_e.
$

We now rewrite the above expression in terms of $\forests_2(G;S)$.
Observe in \eqref{eq:45} that the deletion $T \setminus e$ is an $(S,*)$-rooted spanning forest of $G$, and that the corresponding weights satisfy
\[
	w({F}) = \alpha_e \cdot w({T}) \qquad\text{if}\qquad F = T \setminus e.
\]
Note that $F = T \setminus e$ is equivalent to $T = F \cup e$, and in particular this only occurs when we choose the edge $e$ to be in the floating boundary $\partial F(*)$.

Thus
\begin{align*}
	(\star) 
	&= \sum_{F \in \forests_2} w({F}) (2 - \outdeg(F, *)) \times \\
	& \qquad\qquad\qquad   \sum_{e \in \partial F(*)} \Big( \one(\delta(e; v, \pi_{(F \cup e)}(e)) = 1) - \one(\delta(e; v, \pi_{(F \cup e)}(e)) = 0) \Big) .
\end{align*}
Now for any $e \not\in F$, let $\delta(e; v, F(*)) = \delta(e; v, x)$ for any $x \in F(*)$, i.e.
\[
	\delta(e; v, F(*)) = \begin{cases}
	1 &\text{if $e$ lies on a path from $v$ to $F(*)$}, \\
	0 &\text{otherwise}.
	\end{cases}
\]
The condition that $\delta(e; v, \pi_{(F \cup e)}(e)) = 1$ (respectively, $\delta(e; v, \pi_{(F \cup e)}(e)) = 0$)
is equivalent to ${\delta(e; v, F(*)) = 0}$ (respectively, ${\delta(e; v, F(*)) = 1}$).
For an illustration, compare Figures~\ref{fig:delta-floating-1} and \ref{fig:delta-floating-2}.
Therefore
\begin{multline}\label{eq:penult}
	(\star) 
	= \sum_{F \in \forests_2} w({F}) (2 - \outdeg(F, *)) \times  \\
	\sum_{e \in \partial F(*)} \Big( \one(\delta(e; v, F(*)) = 0) - \one(\delta(e; v, F(*)) = 1) \Big) .
\end{multline}

Finally, we observe that for any forest $F$ in $\forests_2(G;S)$,
there is exactly one edge $e$ in the boundary $\partial F(*)$ of the floating component which satisfies $\delta(e; v, F(*)) = 1$, namely the unique boundary edge on the path from the floating component $F(*)$ to $v$.
Hence
\begin{align*}
	&\#\{e \in \partial F(*) \colon \delta(e;v, F(*)) = 1 \} = 1,
\qquad\text{and}\qquad \\
	&\#\{e \in \partial F(*) \colon \delta(e;v, F(*)) = 0 \} = \outdeg(F,*) - 1 .
\end{align*}
Thus, the previous expression \eqref{eq:penult} simplifies as
\begin{align*}
	(\star) 
	&= - \!\sum_{F \in \forests_2} w({F}) (2 - \outdeg(F,*))^2 
\end{align*}
as desired.
\end{proof}

\begin{figure}[h]
	\begin{minipage}{0.45\textwidth}
		\centering
\begin{tikzpicture}[scale=0.6]
	\coordinate (1) at (0,0);
	\coordinate (A) at (1,0);
	\coordinate (B) at (2,0);
	\coordinate (C) at (3,0);
	\coordinate (D) at (4,0);
	\coordinate (2) at (5,0);
	\coordinate (3) at (2,1);
	
	\foreach \c in {1,2,3,A,B,C,D} {
		\fill (\c) circle (2pt);
	}
	\foreach \c in {1,2,3} {
		\draw (\c) circle (3pt);
	}

	\draw (1) -- (A);
	\draw[dotted] (A) -- node[below] {$e$} (B);
	\draw (B) -- (C);
	\draw[dotted] (C) -- (D);
	\draw (D) -- (2);
	\draw[dotted] (B) -- (3);

	\node[above=0.1] at (2) {$v$};
	\node[below=0.2] at (C) {$F(*)$};

	\draw[color=red, line width=12pt, cap=round, opacity=0.1] (B) -- (C);
\end{tikzpicture}
\end{minipage}
\begin{minipage}{0.45\textwidth}
	\centering
	\begin{tikzpicture}[scale=0.6]
		\coordinate (1) at (0,0);
		\coordinate (A) at (1,0);
		\coordinate (B) at (2,0);
		\coordinate (C) at (3,0);
		\coordinate (D) at (4,0);
		\coordinate (2) at (5,0);
		\coordinate (3) at (2,1);
		
		\foreach \c in {1,2,3,A,B,C,D} {
			\fill (\c) circle (2pt);
		}
		\foreach \c in {1,2,3} {
			\draw (\c) circle (3pt);
		}

		\draw (1) -- (A);
		\draw (A) -- node[below] {$e$} (B);
		\draw (B) -- (C);
		\draw[dotted] (C) -- (D);
		\draw (D) -- (2);
		\draw[dotted] (B) -- (3);

		\node[above=0.1] at (2) {$v$};
		\node[left] at (1) {$\pi_T(e)$};
	\end{tikzpicture}
\end{minipage}
\caption{A forest $F \in \forests_2(G; S)$ and edge $e \in \partial F(*)$,
 with $\delta(e; v, F(*)) = 0$ (left). 
 After adding $e$ to $F$ to obtain $T = F \cup e \in \forests_1(G; S)$,
 we have $\delta(e; v, \pi_T(e)) = 1$ (right).}
\label{fig:delta-floating-1}
\end{figure}

\begin{figure}[h]
	\begin{minipage}{0.45\textwidth}
		\centering
		\begin{tikzpicture}[scale=0.6]
			\coordinate (1) at (0,0);
			\coordinate (A) at (1,0);
			\coordinate (B) at (2,0);
			\coordinate (C) at (3,0);
			\coordinate (D) at (4,0);
			\coordinate (2) at (5,0);
			\coordinate (3) at (2,1);
			
			\foreach \c in {1,2,3,A,B,C,D} {
				\fill (\c) circle (2pt);
			}
			\foreach \c in {1,2,3} {
				\draw (\c) circle (3pt);
			}
		
			\draw (1) -- (A);
			\draw[dotted] (A) -- (B);
			\draw (B) -- (C);
			\draw[dotted] (C) -- node[below] {$e$} (D);
			\draw (D) -- (2);
			\draw[dotted] (B) -- (3);
		
			\node[above=0.1] at (2) {$v$};
			\node[below=0.2] at (B) {$F(*)$};
			
			\draw[color=red, line width=12pt, cap=round, opacity=0.1] (B) -- (C);
		\end{tikzpicture}
	\end{minipage}
	\begin{minipage}{0.45\textwidth}
		\centering
		\begin{tikzpicture}[scale=0.6]
			\coordinate (1) at (0,0);
			\coordinate (A) at (1,0);
			\coordinate (B) at (2,0);
			\coordinate (C) at (3,0);
			\coordinate (D) at (4,0);
			\coordinate (2) at (5,0);
			\coordinate (3) at (2,1);
			
			\foreach \c in {1,2,3,A,B,C,D} {
				\fill (\c) circle (2pt);
			}
			\foreach \c in {1,2,3} {
				\draw (\c) circle (3pt);
			}

			\draw (1) -- (A);
			\draw[dotted] (A) -- (B);
			\draw (B) -- (C);
			\draw (C) -- node[below] {$e$} (D);
			\draw (D) -- (2);
			\draw[dotted] (B) -- (3);

			\node[above=0.1] at (2) {$v$};
			\node[below right] at (2) {$\pi_T(e)$};
		\end{tikzpicture}
	\end{minipage}
\caption{A forest $F \in \forests_2(G; S)$ and edge $e \in \partial F(*)$, with $\delta(e; v, F(*)) = 1$ (left). 
After adding $e$ to $F$ to obtain $T = F \cup e \in \forests_1(G; S)$, we have $\delta(e; v, \pi_T(e)) = 0$ (right).}
\label{fig:delta-floating-2}
\end{figure}
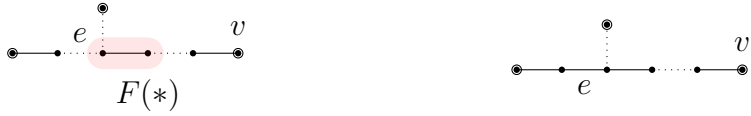

\subsection{Principal minors of $D$}

Finally we can prove our main theorems, giving combinatorial expressions for $\det D[S]$.
The proof uses the argument outlined in \S\ref{sec:proof_outline}.
\begin{proof}[Proof of Theorem~\ref{thm:w-main}]
First, suppose $|S| = 1$.
Then $D[S]$ is the zero matrix, and we must show that the right-hand side is zero.
Since $G$ is a tree, $\trees(G; \{v\})$ consists of the tree $G$ itself, with weight $w({G}) = 1$.
Moreover, the forests in $\forests_2(G; \{v\})$ are precisely $G \setminus e$ for every $e \in E$, and for $F = G \setminus e$ we have 
$w({F}) = \alpha_e$ and
$(\outdeg(F, *) - 2)^2 = 1$.
This shows that the right-hand side of \eqref{eq:w-main} is zero.

Now suppose $|S| \geq 2$.
Lemma~\ref{lem:distance-sub-nonsingular} shows $D[S]$ is nonsingular, so we may use the inverse matrix identity
\begin{equation}\label{eq:cof-invsere-sum}
	\bone^\tr D[S]^{-1} \bone =
	\frac{\cof D[S]}{\det D[S]}. 
\end{equation}
Let $\boldm = \boldm(G; S)$ denote the equilibrium vector \eqref{eq:m-vector}.
By 
Proposition~\ref{prop:m-sum} (a) and Theorem~\ref{thm:distance-sub-cof},
\[
	\bone^\tr \boldm 
	= 2 \sum_{T \in \trees(G;S)} w({T})
	= \frac{\cof D[S]}{(-1)^{1 - |S|} 2^{2 - |S|} }.
\]
Theorem~\ref{thm:m-distance-product} states that $D[S] \boldm = \lambda \bone$
for some constant $\lambda$,
which is nonzero since $D[S]$ is invertible (see Lemma~\ref{lem:distance-sub-nonsingular}) and $\boldm$ is nonzero (see Proposition~\ref{prop:m-sum} (b)).
Hence
\begin{equation}\label{eq:inverse-sum-entries}
	\bone^\tr D[S]^{-1} \bone
	= \lambda^{-1} \bone^\tr \boldm
	= \frac{\cof D[S]}{(-1)^{|S| - 1} 2^{|S| - 1} \lambda} .
\end{equation}
Comparing \eqref{eq:cof-invsere-sum} with \eqref{eq:inverse-sum-entries} gives the desired result,
$\det D[S] = (-1)^{|S| - 1} 2^{|S| - 1} \lambda$.
\end{proof}

\begin{proof}[Proof of Theorem~\ref{thm:main}]
Set all lengths $\alpha_e$ to $1$ in Theorem~\ref{thm:w-main}.
In this case, the weights $w({T}) = 1$ and $w({F}) = 1$ for all forests $T$ and $F$,
and 
\[
	\sum_{e \in E} \alpha_e = n - 1, 
	\qquad \sum_{T \in \trees(G;S)} w({T}) = \kappa(G; S).
	\qedhere
\]
\end{proof}

\section{Characteristic polynomial coefficients}\label{sec:char-poly}

In this section we prove Theorem~\ref{thm:dist-poly-coeff}, on the coefficients of the characteristic polynomial of the distance matrix.
For simplicity, we have stated the result for trees with unit edge lengths. 
It will be clear from the proof below that a similar result holds for trees with arbitrary edge lengths. 

We make use of the following notation.
For a positive integer $k$, we let $\forests_k(G)$ denote the set of spanning forests of $G$ that contain exactly $k$ connected components.
If a forest $F$ has connected components $F = C_1 \sqcup \cdots \sqcup C_k$, then we let
\[
	\pi(F) := \prod_{i = 1}^k |V(C_i)|.
\]
Using this notation, recall that Theorem~\ref{thm:dist-poly-coeff} states that the $\lambda^{n - k}$-coefficient of the characteristic polynomial of $D$ is
\begin{multline}\label{eq:dist-poly-coeff}
	\delta_{n - k}
	= (-1)^{n - 1} 2^{k - 2} \bigg((n - 1) \sum_{\substack{F \in \mathcal F_{k}(G) }} \pi(F) 
	\\ - \sum_{\substack{F \in \mathcal F_{k + 1}(G) \\ F = C_1 \sqcup \cdots \sqcup C_{k + 1}}} \!\!\! \pi(F) \sum_{j = 1}^{k + 1}  \frac{(\outdeg(C_j) - 2)^2}{|V(C_j)|}\bigg) .
\end{multline}
\begin{proof}[Proof of Theorem~\ref{thm:dist-poly-coeff}]
Let $\det(D - \lambda I) = \sum_{k = 0}^n \delta_{n - k} \lambda^{n - k}$.
It is a standard result in matrix algebra that
$\displaystyle \delta_{n - k} = (-1)^{n - k} \sum_{|S| = k} \det D[S]$.
Using Theorem~\ref{thm:main},
\begin{align}\label{eq:minor-sum}
	\delta_{n - k} 
	&= (-1)^{n - 1} 2^{k - 2} \left( (n - 1)\sum_{|S| = k} \kappa(G; S) - \sum_{|S| = k} \sum_{\forests_2(G; S)} (\outdeg(F, *) - 2)^2 \right) .
\end{align}
Then, observe that
\begin{align}
	\sum_{|S| = k} \kappa(G; S) &= \sum_{\substack{S \subset V \\ |S| = k}} |\{F \in \forests_k(G) : \text{$F$ is rooted by $S$}\}| \label{eq:c-pf-term-1}\\
	&= \sum_{F \in \forests_k(G)} |\{S \subset V : \text{$F$ is rooted by $S$}\}|
	= \sum_{F \in \forests_k(G)} \pi(F), \nonumber
\end{align}
since $\pi(F) := \prod_{i = 1}^k |V(C_i)|$ is the number of ways of rooting the forest $F$ with components $C_1 \sqcup \cdots \sqcup C_k$.
Next, the summation of $ (\outdeg(F, *) - 2)^2$-terms over $\forests_2(G; S)$ can be similarly rearranged to
\begin{equation}\label{eq:c-pf-term-2}
	\sum_{|S| = k} \sum_{\forests_2(G; S)} (\outdeg(F, *) - 2)^2 = 
	\sum_{\substack{F \in \forests_{k + 1}(G) \\ F = C_1 \sqcup \cdots \sqcup C_{k + 1}}} \sum_{j = 1}^{k + 1} \frac{\pi(F)}{|V(C_j)|} \cdot  (\outdeg(C_j) - 2)^2
\end{equation}
since each forest in $\forests_2(G; S)$ has $k + 1$ components and $\pi(F) / {|V(C_j)|} $ is the number of ways to mark $F$ as an $(S, *)$-rooted forest in which $C_j$ is the floating component.
The desired result \eqref{eq:dist-poly-coeff} now follows from \eqref{eq:minor-sum}, \eqref{eq:c-pf-term-1}, and \eqref{eq:c-pf-term-2}.
\end{proof}

\begin{rmk}
In~\cite{graham-szemeredi}, it is shown that, when expressing $\delta_i$ as a linear combination of subforest counts, the resulting subforest coefficients are unique.
At first glance, this seems to contradict the existence of both \cite{graham-lovasz} and Theorem~\ref{thm:dist-poly-coeff} as distinct expressions for $\delta_i$ in terms of subforest counts.

This apparent contradiction is resolved by noting that, in \eqref{eq:dist-poly-coeff}, the term $\outdeg(C_j)$ depends not only on the forest $F$, but on the embedding in the ambient tree $F \subset G$.
The forests in \eqref{eq:dist-poly-coeff} also may have isolated vertices, which is disallowed by \cite{graham-szemeredi}.

\end{rmk}

\section{Examples}
\label{sec:examples}

\begin{eg}
Suppose $G$ is the tree with unit edge lengths shown in Figure~\ref{fig:tree-5}, with five leaf vertices and three internal vertices.
Let $S=\{1,2,3,4,5\}$ be the set of leaf vertices. 
The corresponding distance submatrix is
\[
	D[S] = \begin{pmatrix}
	0 & 2 & 3 & 4 & 4 \\
	2 & 0 & 3 & 4 & 4 \\
	3 & 3 & 0 & 3 & 3 \\
	4 & 4 & 3 & 0 & 2 \\
	4 & 4 & 3 & 2 & 0
	\end{pmatrix},
\]
whose determinant is $864$.

\begin{figure}[h]
\centering
\begin{tikzpicture}[scale=0.6]
	\coordinate (1) at (-2,0.7);
	\coordinate (2) at (-2,-0.6);
	\coordinate (3) at (0,-1);
	\coordinate (4) at (2,-0.7);
	\coordinate (5) at (2,0.7);
	\coordinate (A) at (-1,0);
	\coordinate (B) at (0,0);
	\coordinate (C) at (1,0);

	\node[left=2pt] at (1) {$1$};
	\node[left=2pt] at (2) {$2$};
	\node[left=2pt] at (3) {$3$};
	\node[right=2pt] at (4) {$4$};
	\node[right=2pt] at (5) {$5$};

	\draw (B) -- (A) -- (1);
	\draw (A) -- (2);
	\draw (3) -- (B) -- (C) -- (4);
	\draw (C) -- (5);
	
	\foreach \c in {1,2,3,4,5,A,B,C} {
		\fill[black] (\c) circle (3pt);
	}
\end{tikzpicture}
\caption{Tree with five leaves.}
\label{fig:tree-5}
\end{figure}

The tree $G$ has $7$ edges and $21$ $S$-rooted spanning forests.
There are $19$ $(S, *)$-rooted spanning forests. 
Of the floating components in these forests, $14$ have outdegree three, $4$ have outdegree four, and $1$ has outdegree five.
By Theorem~\ref{thm:main},
\[
	\det D[S] 
	= (-1)^4\, 2^3 \left( 7 \cdot 21 - (14 \cdot 1^2 + 4 \cdot 2^2 + 1 \cdot 3^2) \right) = 864.
\]
\end{eg}

\begin{eg}
Suppose $G$ is a tree consisting of three edges joined at a central vertex.
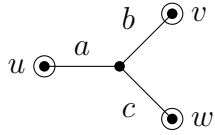
\[
\begin{tikzpicture}
	\coordinate (u) at (-1,0);
	\coordinate (v) at (0.7,0.7);
	\coordinate (w) at (0.7,-0.7);
	\coordinate (o) at (0,0);

	\draw (o) -- node[above] {$a$} (u);
	\draw (v) -- node[above left] {$b$} (o)
		-- node[below left] {$c$} (w);

	\foreach \c in {u, v, w, o} {
		\fill (\c) circle (2pt);
	}
	\foreach \c in {u, v, w} {
		\draw (\c) circle (4pt);
	}
	
	\node[left=3pt] at (u) {$u$};
	\node[right=3pt] at (v) {$v$};
	\node[right=3pt] at (w) {$w$};
\end{tikzpicture}
\]
Suppose $S$ consists of the leaf vertices $ \{u,v,w\}$.
Then 
\[
	D[S] = \begin{pmatrix}
	0 & a + b & a + c \\
	a + b & 0 & b + c \\
	a + c & b + c & 0
	\end{pmatrix}
\]
which has determinant
\[\det D[S] = 2(a+b)(a+c)(b+c) 
	= 2\left( (a+b+c)(ab + ac + bc) - abc \right).
\]
The equilibrium vector (which satisfies $D[S] \boldm = \lambda \bone$) in this example is 
\[
	\boldm = \begin{pmatrix} 
	ab + ac & 
	ab + bc &
	ac + bc 
	\end{pmatrix}^\tr.
\]
\end{eg}

\begin{eg}
Suppose $G$ is the tree with unit edge lengths shown below, with five leaf vertices.
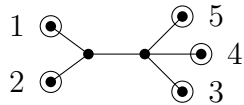
\[
\begin{tikzpicture}[scale=1.0]
	\coordinate (1) at (-1,0.37);
	\coordinate (2) at (-1,-0.37);
	\coordinate (3) at (0.75,-0.5);
	\coordinate (4) at (1,0);
	\coordinate (5) at (0.75,0.5);
	\coordinate (A) at (-0.5,0);
	\coordinate (B) at (0.25,0);
	
	\foreach \c in {1,2,3,4,5,A,B} {
		\fill (\c) circle (2pt);
	}
	\foreach \c in {1,2,3,4,5} {
		\draw (\c) circle (4pt);
	}

	\draw (A) -- (1);
	\draw (A) -- (2);
	\draw (B) -- (3);
	\draw (B) -- (4);
	\draw (B) -- (5);
	\draw (A) -- (B);
	
	\foreach \c/\d in {1/left,2/left,3/right,4/right,5/right} {
		\node[\d=0.2] at (\c) {$\c$};
	}
\end{tikzpicture}
\]
Let $S$ denote the set of five leaf vertices. Then
\[
	D[S] = \begin{pmatrix}
	0 & 2 & 3 & 3 & 3 \\
	2 & 0 & 3 & 3 & 3 \\
	3 & 3 & 0 & 2 & 2 \\
	3 & 3 & 2 & 0 & 2 \\
	3 & 3 & 2 & 2 & 0
	\end{pmatrix}.
\]
There are $11$ forests in $\trees(G;S)$:
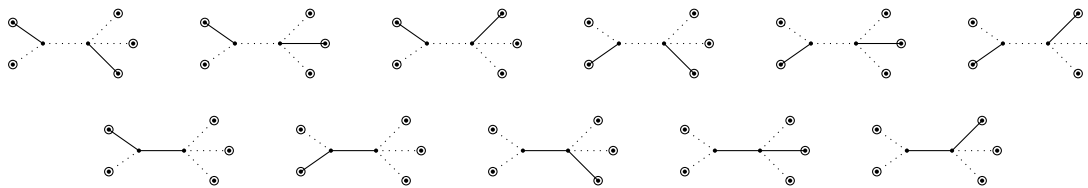
\[
\begin{tikzpicture}[scale=0.4]
	\coordinate (1) at (-2,0.7);
	\coordinate (2) at (-2,-0.7);
	\coordinate (3) at (1.5,-1);
	\coordinate (4) at (2,0);
	\coordinate (5) at (1.5,1);
	\coordinate (A) at (-1,0);
	\coordinate (B) at (0.5,0);
	
	\foreach \c in {1,2,3,4,5,A,B} {
		\fill (\c) circle (2pt);
	}
	\foreach \c in {1,2,3,4,5} {
		\draw (\c) circle (4pt);
	}

	\draw (A) -- (1);
	\draw[dotted] (A) -- (2);
	\draw (B) -- (3);
	\draw[dotted] (B) -- (4);
	\draw[dotted] (B) -- (5);
	\draw[dotted] (A) -- (B);
\end{tikzpicture}
\qquad
\begin{tikzpicture}[scale=0.4]
	\coordinate (1) at (-2,0.7);
	\coordinate (2) at (-2,-0.7);
	\coordinate (3) at (1.5,-1);
	\coordinate (4) at (2,0);
	\coordinate (5) at (1.5,1);
	\coordinate (A) at (-1,0);
	\coordinate (B) at (0.5,0);
	
	\foreach \c in {1,2,3,4,5,A,B} {
		\fill (\c) circle (2pt);
	}
	\foreach \c in {1,2,3,4,5} {
		\draw (\c) circle (4pt);
	}

	\draw (A) -- (1);
	\draw[dotted] (A) -- (2);
	\draw[dotted] (B) -- (3);
	\draw (B) -- (4);
	\draw[dotted] (B) -- (5);
	\draw[dotted] (A) -- (B);
\end{tikzpicture}
\qquad
\begin{tikzpicture}[scale=0.4]
	\coordinate (1) at (-2,0.7);
	\coordinate (2) at (-2,-0.7);
	\coordinate (3) at (1.5,-1);
	\coordinate (4) at (2,0);
	\coordinate (5) at (1.5,1);
	\coordinate (A) at (-1,0);
	\coordinate (B) at (0.5,0);
	
	\foreach \c in {1,2,3,4,5,A,B} {
		\fill (\c) circle (2pt);
	}
	\foreach \c in {1,2,3,4,5} {
		\draw (\c) circle (4pt);
	}

	\draw (A) -- (1);
	\draw[dotted] (A) -- (2);
	\draw[dotted] (B) -- (3);
	\draw[dotted] (B) -- (4);
	\draw (B) -- (5);
	\draw[dotted] (A) -- (B);
\end{tikzpicture}
\qquad
\begin{tikzpicture}[scale=0.4]
	\coordinate (1) at (-2,0.7);
	\coordinate (2) at (-2,-0.7);
	\coordinate (3) at (1.5,-1);
	\coordinate (4) at (2,0);
	\coordinate (5) at (1.5,1);
	\coordinate (A) at (-1,0);
	\coordinate (B) at (0.5,0);
	
	\foreach \c in {1,2,3,4,5,A,B} {
		\fill (\c) circle (2pt);
	}
	\foreach \c in {1,2,3,4,5} {
		\draw (\c) circle (4pt);
	}

	\draw[dotted] (A) -- (1);
	\draw (A) -- (2);
	\draw (B) -- (3);
	\draw[dotted] (B) -- (4);
	\draw[dotted] (B) -- (5);
	\draw[dotted] (A) -- (B);
\end{tikzpicture}
\qquad
\begin{tikzpicture}[scale=0.4]
	\coordinate (1) at (-2,0.7);
	\coordinate (2) at (-2,-0.7);
	\coordinate (3) at (1.5,-1);
	\coordinate (4) at (2,0);
	\coordinate (5) at (1.5,1);
	\coordinate (A) at (-1,0);
	\coordinate (B) at (0.5,0);
	
	\foreach \c in {1,2,3,4,5,A,B} {
		\fill (\c) circle (2pt);
	}
	\foreach \c in {1,2,3,4,5} {
		\draw (\c) circle (4pt);
	}

	\draw[dotted] (A) -- (1);
	\draw (A) -- (2);
	\draw[dotted] (B) -- (3);
	\draw (B) -- (4);
	\draw[dotted] (B) -- (5);
	\draw[dotted] (A) -- (B);
\end{tikzpicture}
\qquad
\begin{tikzpicture}[scale=0.4]
	\coordinate (1) at (-2,0.7);
	\coordinate (2) at (-2,-0.7);
	\coordinate (3) at (1.5,-1);
	\coordinate (4) at (2,0);
	\coordinate (5) at (1.5,1);
	\coordinate (A) at (-1,0);
	\coordinate (B) at (0.5,0);
	
	\foreach \c in {1,2,3,4,5,A,B} {
		\fill (\c) circle (2pt);
	}
	\foreach \c in {1,2,3,4,5} {
		\draw (\c) circle (4pt);
	}

	\draw[dotted] (A) -- (1);
	\draw (A) -- (2);
	\draw[dotted] (B) -- (3);
	\draw[dotted] (B) -- (4);
	\draw (B) -- (5);
	\draw[dotted] (A) -- (B);
\end{tikzpicture}
\]
\[
\begin{tikzpicture}[scale=0.4]
	\coordinate (1) at (-2,0.7);
	\coordinate (2) at (-2,-0.7);
	\coordinate (3) at (1.5,-1);
	\coordinate (4) at (2,0);
	\coordinate (5) at (1.5,1);
	\coordinate (A) at (-1,0);
	\coordinate (B) at (0.5,0);
	
	\foreach \c in {1,2,3,4,5,A,B} {
		\fill (\c) circle (2pt);
	}
	\foreach \c in {1,2,3,4,5} {
		\draw (\c) circle (4pt);
	}

	\draw (A) -- (1);
	\draw[dotted] (A) -- (2);
	\draw[dotted] (B) -- (3);
	\draw[dotted] (B) -- (4);
	\draw[dotted] (B) -- (5);
	\draw (A) -- (B);
\end{tikzpicture}
\qquad
\begin{tikzpicture}[scale=0.4]
	\coordinate (1) at (-2,0.7);
	\coordinate (2) at (-2,-0.7);
	\coordinate (3) at (1.5,-1);
	\coordinate (4) at (2,0);
	\coordinate (5) at (1.5,1);
	\coordinate (A) at (-1,0);
	\coordinate (B) at (0.5,0);
	
	\foreach \c in {1,2,3,4,5,A,B} {
		\fill (\c) circle (2pt);
	}
	\foreach \c in {1,2,3,4,5} {
		\draw (\c) circle (4pt);
	}

	\draw[dotted] (A) -- (1);
	\draw (A) -- (2);
	\draw[dotted] (B) -- (3);
	\draw[dotted] (B) -- (4);
	\draw[dotted] (B) -- (5);
	\draw (A) -- (B);
\end{tikzpicture}
\qquad
\begin{tikzpicture}[scale=0.4]
	\coordinate (1) at (-2,0.7);
	\coordinate (2) at (-2,-0.7);
	\coordinate (3) at (1.5,-1);
	\coordinate (4) at (2,0);
	\coordinate (5) at (1.5,1);
	\coordinate (A) at (-1,0);
	\coordinate (B) at (0.5,0);
	
	\foreach \c in {1,2,3,4,5,A,B} {
		\fill (\c) circle (2pt);
	}
	\foreach \c in {1,2,3,4,5} {
		\draw (\c) circle (4pt);
	}

	\draw[dotted] (A) -- (1);
	\draw[dotted] (A) -- (2);
	\draw (B) -- (3);
	\draw[dotted] (B) -- (4);
	\draw[dotted] (B) -- (5);
	\draw (A) -- (B);
\end{tikzpicture}
\qquad
\begin{tikzpicture}[scale=0.4]
	\coordinate (1) at (-2,0.7);
	\coordinate (2) at (-2,-0.7);
	\coordinate (3) at (1.5,-1);
	\coordinate (4) at (2,0);
	\coordinate (5) at (1.5,1);
	\coordinate (A) at (-1,0);
	\coordinate (B) at (0.5,0);
	
	\foreach \c in {1,2,3,4,5,A,B} {
		\fill (\c) circle (2pt);
	}
	\foreach \c in {1,2,3,4,5} {
		\draw (\c) circle (4pt);
	}

	\draw[dotted] (A) -- (1);
	\draw[dotted] (A) -- (2);
	\draw[dotted] (B) -- (3);
	\draw (B) -- (4);
	\draw[dotted] (B) -- (5);
	\draw (A) -- (B);
\end{tikzpicture}
\qquad
\begin{tikzpicture}[scale=0.4]
	\coordinate (1) at (-2,0.7);
	\coordinate (2) at (-2,-0.7);
	\coordinate (3) at (1.5,-1);
	\coordinate (4) at (2,0);
	\coordinate (5) at (1.5,1);
	\coordinate (A) at (-1,0);
	\coordinate (B) at (0.5,0);
	
	\foreach \c in {1,2,3,4,5,A,B} {
		\fill (\c) circle (2pt);
	}
	\foreach \c in {1,2,3,4,5} {
		\draw (\c) circle (4pt);
	}

	\draw[dotted] (A) -- (1);
	\draw[dotted] (A) -- (2);
	\draw[dotted] (B) -- (3);
	\draw[dotted] (B) -- (4);
	\draw (B) -- (5);
	\draw (A) -- (B);
\end{tikzpicture}
\]
There are $6$ forests in $\forests_2(G;S)$:

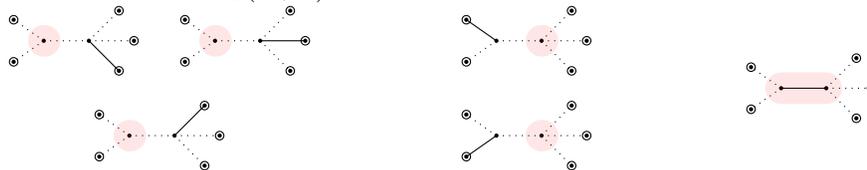
\begin{minipage}{0.4\textwidth}
	\centering
	\begin{tikzpicture}[scale=0.4]
		\coordinate (1) at (-2,0.7);
		\coordinate (2) at (-2,-0.7);
		\coordinate (3) at (1.5,-1);
		\coordinate (4) at (2,0);
		\coordinate (5) at (1.5,1);
		\coordinate (A) at (-1,0);
		\coordinate (B) at (0.5,0);
		
		\foreach \c in {1,2,3,4,5,A,B} {
			\fill (\c) circle (2pt);
		}
		\foreach \c in {1,2,3,4,5} {
			\draw (\c) circle (4pt);
		}

		\draw[dotted] (A) -- (1);
		\draw[dotted] (A) -- (2);
		\draw (B) -- (3);
		\draw[dotted] (B) -- (4);
		\draw[dotted] (B) -- (5);
		\draw[dotted] (A) -- (B);

		\draw[color=red, line width=12pt, cap=round, opacity=0.1] (A) -- (A);		
	\end{tikzpicture}
	\quad
	\begin{tikzpicture}[scale=0.4]
		\coordinate (1) at (-2,0.7);
		\coordinate (2) at (-2,-0.7);
		\coordinate (3) at (1.5,-1);
		\coordinate (4) at (2,0);
		\coordinate (5) at (1.5,1);
		\coordinate (A) at (-1,0);
		\coordinate (B) at (0.5,0);
		
		\foreach \c in {1,2,3,4,5,A,B} {
			\fill (\c) circle (2pt);
		}
		\foreach \c in {1,2,3,4,5} {
			\draw (\c) circle (4pt);
		}

		\draw[dotted] (A) -- (1);
		\draw[dotted] (A) -- (2);
		\draw[dotted] (B) -- (3);
		\draw (B) -- (4);
		\draw[dotted] (B) -- (5);
		\draw[dotted] (A) -- (B);

		\draw[color=red, line width=12pt, cap=round, opacity=0.1] (A) -- (A);		
	\end{tikzpicture} 
	\\[0.3cm]
	\begin{tikzpicture}[scale=0.4]
		\coordinate (1) at (-2,0.7);
		\coordinate (2) at (-2,-0.7);
		\coordinate (3) at (1.5,-1);
		\coordinate (4) at (2,0);
		\coordinate (5) at (1.5,1);
		\coordinate (A) at (-1,0);
		\coordinate (B) at (0.5,0);
		
		\foreach \c in {1,2,3,4,5,A,B} {
			\fill (\c) circle (2pt);
		}
		\foreach \c in {1,2,3,4,5} {
			\draw (\c) circle (4pt);
		}

		\draw[dotted] (A) -- (1);
		\draw[dotted] (A) -- (2);
		\draw[dotted] (B) -- (3);
		\draw[dotted] (B) -- (4);
		\draw (B) -- (5);
		\draw[dotted] (A) -- (B);

		\draw[color=red, line width=12pt, cap=round, opacity=0.1] (A) -- (A);		
	\end{tikzpicture}
\end{minipage}
\begin{minipage}{0.25\textwidth}
\centering
	\begin{tikzpicture}[scale=0.4]
		\coordinate (1) at (-2,0.7);
		\coordinate (2) at (-2,-0.7);
		\coordinate (3) at (1.5,-1);
		\coordinate (4) at (2,0);
		\coordinate (5) at (1.5,1);
		\coordinate (A) at (-1,0);
		\coordinate (B) at (0.5,0);
		
		\foreach \c in {1,2,3,4,5,A,B} {
			\fill (\c) circle (2pt);
		}
		\foreach \c in {1,2,3,4,5} {
			\draw (\c) circle (4pt);
		}

		\draw (A) -- (1);
		\draw[dotted] (A) -- (2);
		\draw[dotted] (B) -- (3);
		\draw[dotted] (B) -- (4);
		\draw[dotted] (B) -- (5);
		\draw[dotted] (A) -- (B);

		\draw[color=red, line width=12pt, cap=round, opacity=0.1] (B) -- (B);		
	\end{tikzpicture}
	\\[0.3cm]
	\begin{tikzpicture}[scale=0.4]
		\coordinate (1) at (-2,0.7);
		\coordinate (2) at (-2,-0.7);
		\coordinate (3) at (1.5,-1);
		\coordinate (4) at (2,0);
		\coordinate (5) at (1.5,1);
		\coordinate (A) at (-1,0);
		\coordinate (B) at (0.5,0);
		
		\foreach \c in {1,2,3,4,5,A,B} {
			\fill (\c) circle (2pt);
		}
		\foreach \c in {1,2,3,4,5} {
			\draw (\c) circle (4pt);
		}
	
		\draw[dotted] (A) -- (1);
		\draw (A) -- (2);
		\draw[dotted] (B) -- (3);
		\draw[dotted] (B) -- (4);
		\draw[dotted] (B) -- (5);
		\draw[dotted] (A) -- (B);

		\draw[color=red, line width=12pt, cap=round, opacity=0.1] (B) -- (B);		
	\end{tikzpicture}
\end{minipage}
\begin{minipage}{0.25\textwidth}
\centering
	\begin{tikzpicture}[scale=0.4]
		\coordinate (1) at (-2,0.7);
		\coordinate (2) at (-2,-0.7);
		\coordinate (3) at (1.5,-1);
		\coordinate (4) at (2,0);
		\coordinate (5) at (1.5,1);
		\coordinate (A) at (-1,0);
		\coordinate (B) at (0.5,0);
		
		\foreach \c in {1,2,3,4,5,A,B} {
			\fill (\c) circle (2pt);
		}
		\foreach \c in {1,2,3,4,5} {
			\draw (\c) circle (4pt);
		}

		\draw[dotted] (A) -- (1);
		\draw[dotted] (A) -- (2);
		\draw[dotted] (B) -- (3);
		\draw[dotted] (B) -- (4);
		\draw[dotted] (B) -- (5);
		\draw (A) -- (B);

		\draw[color=red, line width=12pt, cap=round, opacity=0.1] (A) -- (B);		
	\end{tikzpicture}
\end{minipage}

\noindent
Out of the floating components of forests in $\forests_2(G; S)$, 3 have outdegree three, 2 have outdegree four, and 1 has outdegree five.

The determinant of the distance submatrix is
\[
	\det D[S] = 368	
	= (-1)^4 2^3 \left( 6 \cdot 11 - (3 \cdot 1^2 + 2 \cdot 2^2 + 1 \cdot 3^2) \right),
\]
and the equilibrium vector is $\boldm = \begin{pmatrix}
5 & 5 & 4 & 4 & 4
\end{pmatrix}^\tr.$
\end{eg}

\begin{eg}
Suppose $G$ is the tree with edge lengths shown in Figure~\ref{fig:tree-4}, with four leaf vertices and two internal vertices.
Let $S$ denote the set of four leaf vertices.
\begin{figure}[h]
	\centering
	\begin{tikzpicture}
		\coordinate (1) at (-1.3,0.47);
		\coordinate (2) at (-1.3,-0.47);
		\coordinate (A) at (-0.5,0);
		\coordinate (B) at (0.5,0);
		\coordinate (3) at (1.3,0.47);
		\coordinate (4) at (1.3,-0.47);
	
		\draw (1) -- node[above] {$a$} (A) 
			-- node[below] {$b$} (2);
		\draw (A) -- node[above] {$c$} (B);
		\draw (3) -- node[above] {$d$} (B) 
			-- node[below] {$e$} (4);
	
		\foreach \c in {1, 2, 3, 4, A, B} {
			\fill[black] (\c) circle (2pt);
		}
		\node[left=2pt] at (1) {$1$};
		\node[left=2pt] at (2) {$2$};
		\node[right=2pt] at (3) {$3$};
		\node[right=2pt] at (4) {$4$};
	\end{tikzpicture}
	\caption{Tree with four leaves, and varying edge lengths.}
	\label{fig:tree-4}
\end{figure}
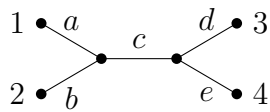
	
Then
\[
	D[S] = \begin{pmatrix}
	0 & a + b & a + c + d & a + c + e \\
	a + b & 0 & b + c + d & b + c + e \\
	a + c + d & b + c + d & 0 & d + e \\
	a + c + e & b + c + e & d + e & 0
	\end{pmatrix}
\]
and the equilibrium vector is
\begin{multline}
	\boldm = 
	abd \begin{pmatrix} 1 \\ 1 \\ 1 \\ -1 \end{pmatrix} 
	+ abe \begin{pmatrix} 1 \\ 1 \\ -1 \\ 1 \end{pmatrix} 
	+ acd \begin{pmatrix} 1 \\ 0 \\ 1 \\ 0 \end{pmatrix} 
	+ ace \begin{pmatrix} 1 \\ 0 \\ 0 \\ 1 \end{pmatrix} 
	+ ade \begin{pmatrix} 1 \\ -1 \\ 1 \\ 1 \end{pmatrix} \\
	+ bcd \begin{pmatrix} 0 \\ 1 \\ 1 \\ 0 \end{pmatrix} 
	+ bce \begin{pmatrix} 0 \\ 1 \\ 0 \\ 1 \end{pmatrix} 
	+ bde \begin{pmatrix} -1 \\ 1 \\ 1 \\ 1 \end{pmatrix} .
\end{multline}

The determinant of $D[S]$ normalized by its sum of cofactors, is
\begin{multline}
	\frac{\det D[S]}{\cof D[S]} 
	= \frac{1}{2} \Big( (a + b + c + d + e)  \\
	 - \frac{(abcd + abce + acde + bcde) + 4 (abde)}{(abd + abe + acd + ace + ade + bcd + bce + bde)} \Big).
\end{multline}
\end{eg}

\bibliography{tree-distance-ref} 
\bibliographystyle{abbrv}

\vspace{3mm}
\end{document}